\documentclass[a4paper,11pt]{scrartcl}

\usepackage{graphicx}
\usepackage[utf8]{inputenc}
\usepackage[english]{babel}
\usepackage{amsmath,amssymb,enumerate,url,tabularx,stmaryrd,mathrsfs}

\usepackage{todonotes,comment,afterpage}
\usepackage{subcaption}
\usepackage{framed}
\usepackage{mathtools}
\usepackage{fullpage}
\usepackage[title]{appendix}
\usepackage{nicefrac}

\usepackage{amsthm}
\usepackage{graphicx}
\usepackage{hyperref}

\theoremstyle{plain}
\newtheorem{theorem}{Theorem}[section]
\newtheorem{lemma}[theorem]{Lemma}
\newtheorem{corollary}[theorem]{Corollary}
\newtheorem{proposition}[theorem]{Proposition}

\newtheorem{definition}[theorem]{Definition}

\newtheorem{remark}[theorem]{Remark}

\numberwithin{equation}{section}
\numberwithin{figure}{section}
\numberwithin{table}{section}

\newcommand{\R}{\mathbb{R}}
\newcommand{\C}{\mathbb{C}}
\newcommand{\N}{\mathbb{N}}

\newcommand{\bigO}{\mathcal{O}}

\newcommand{\abs}[1]{\left|#1\right|}
\newcommand{\norm}[1]{\left\|#1\right\|}

\newcommand{\support}{\operatorname{supp}}

\newcommand{\laplace}{\Delta}

\newcommand{\tracejump}[1]{{\left\llbracket \gamma #1 \right \rrbracket}}
\newcommand{\normaljump}[1]{{\left\llbracket \partial_{\nu} #1 \right \rrbracket}}

\newcommand{\normaljumps}[1]{{\llbracket \partial_{\nu} #1  \rrbracket}}

\def\UU{\mathcal{U}}
\def\VV{\mathcal{V}}
\def\YY{\mathcal{Y}}

\def\TT{\mathcal{T}}
\def\SS{\mathcal{S}}

\def\ii{i}

\def\fdiv{\operatorname{div}}

\def\LL{\mathcal{L}}

\def\HHwr{H^1_{\rho}(y^\alpha,\R^d\times\R^+)}
\def\HHwry{H^1_{\rho}(y^\alpha,\R^d \times (0,\YY))}

\def\HHwy{H^1(y^\alpha,\R^d \times (0,\YY))}

\def\WHwr{H^1_{\rho}(y^\alpha,\R^{d}\setminus \Gamma \times (0,\YY))}
\def\WHwrInf{H^1_{\rho}(y^\alpha,\R^{d}\setminus \Gamma \times (0,\infty))}

\def\HH{\mathbb{H}}

\def\VHx{\mathbb{V}_h^{x}}
\def\VHy{\mathbb{V}_h^{y}}
\def\VHl{\mathbb{V}_h^{\lambda}}

\def\trace{\mathrm{tr_0}}

\renewcommand{\Re}{\operatorname{Re}}

\usepackage{tikz}
\usetikzlibrary{shapes}
\usepackage{pgfplots}
\usetikzlibrary{calc}

\usetikzlibrary{positioning}
\usetikzlibrary{arrows}
\usetikzlibrary{calc}
\usetikzlibrary{intersections, pgfplots.fillbetween,patterns}
\usepgfplotslibrary{colorbrewer}

\iftrue 
\newcommand{\includeTikzOrEps}[1]{ \include{figures/#1}}  

\usetikzlibrary{external}
\tikzexternaldisable

\else
\newcommand{\includeTikzOrEps}[1]{\includegraphics{figures_pdf/#1}}
\fi

\title{FEM-BEM coupling in Fractional Diffusion}

\author{Markus Faustmann\footnote{
Institute for Analysis and Scientific Computing, TU Wien, Vienna, Austria,
\texttt{markus.faustmann@tuwien.ac.at}} 
\, and Alexander Rieder\footnote{
Institute for Analysis and Scientific Computing, TU Wien, Vienna, Austria,
\texttt{alexander.rieder@tuwien.ac.at}}}
\date{\today}

\begin{document}
\maketitle
\begin{abstract}
We derive and analyze a fully computable discrete scheme for fractional partial differential equations posed on the full space $\R^d$. Based on a reformulation using the well-known Caffarelli-Silvestre extension, we study a modified variational formulation to obtain well-posedness of the discrete problem. Our scheme is obtained by combining a diagonalization procedure with a reformulation using boundary integral equations and a coupling of finite elements and boundary elements. For our discrete method we present a-priori estimates as well as numerical examples.
\end{abstract}

\section{Introduction}
In this work, we study stationary fractional partial differential equations posed on the full space $\R^d$ with $d=2,3$ of the form 
\begin{align}\label{eq:stat_model_prob}
  \LL^{\beta} u + s u &= f \;  \qquad \text{in $\R^d$},  \qquad \quad   \LL u :=-\fdiv\big( \mathfrak{A} \nabla u\big)
\end{align}
with $s \geq 0$,  and $\beta \in (0,1)$.
Fractional PDEs of this type are oftentimes used to model non-local effects in physics, finance or image processing,  \cite{bv16,collection_of_applications}.

Regarding the formal definition of non-integer powers $\LL^\beta$ of differential operators, there are various different descriptions in literature such as Fourier transformation, semigroup approaches, singular integrals or spectral calculus, see \cite{whatis}. 
A distinct advantage of full-space formulations as in \eqref{eq:stat_model_prob} is that all of these definitions are equivalent, \cite{tendef}, while there are significant differences in the definitions, if one restricts the problem to a bounded domain. 
\bigskip

Nonetheless, there are usually no closed form solutions to these problems available and therefore numerical approximations are used. In order to derive a computable approximation, most numerical methods employ formulations on bounded domains, for which there is a fairly well developed literature. 
We mention the surveys  \cite{bbnos18,whatis} as well as finite element methods for the integral definition of the fractional Laplacian \cite{ab17,abh19,fmk22}, for the spectral definition \cite{NOS15,pde_frac_parabolic}, and semigroup approaches 
\cite{bp15,blp_accretive}. We especially mention the very influential reformulation using the extension approach by Caffarelli and Silvestre~\cite{caffarelli_silvestre} (see also~\cite{stinga_torrea} for a more general setting), which allows to use PDE techniques in the analysis. This approach paired with an $hp$-FEM approach in the extended direction has proven to be an
effective strategy both for elliptic~\cite{mpsv17,tensor_fem,bms20,fmms21,fmms22} as well as
parabolic problems \cite{pde_frac_parabolic,hp_for_heat}.

Many numerical approaches for the full-space formulation, like \cite{akmr21} for the fractional Allen-Cahn equation, rely on truncation of the full-space problem to a bounded domain, which induces an additional truncation error that needs to be investigated. 
A different approach that avoids any truncation errors is the use of a  coupling of finite elements on a truncated domain and boundary elements appearing from a reformulation of the unbounded exterior part as a boundary integral equation.  
We refer to the classical works \cite{jn80,costabel88a,han90} for the one-equation/Johnson-Nédélec coupling and the symmetric coupling for elliptic transmission problems. For the standard Laplacian these methods are well-posed and thoroughly analyzed, \cite{sayas09,steinbach11,affkmp13}.

In this work, we introduce a method for elliptic full-space fractional operators that combines the mentioned Caffarelli-Silvestre extension approach with FEM-BEM coupling techniques. 
More precisely, inspired by \cite{LS09,sayas09}, we reformulate the extension problem as a variational problem, where the solution on a bounded domain and an exterior solution
 in an exotic Hilbert space are sought. Using suitable Poincar\'e inequalities, we show well-posedness of the continuous formulation. 
In order to obtain a computable approximation, we then use the diagonalization procedure of  \cite{tensor_fem}, which leads to a sequence of  Helmholtz-type transmission problems. For those, we employ a standard coupling of FEM and BEM of symmetric type,
as proposed by \cite{costabel88a,han90}. 
Finally, we present an a priori analysis for a discretization with $hp$-FEM in the extended variable.  

This work builds on the recent a priori analysis of \cite{part_one}, where
  the regularity and the decay properties of the analytic full-space solution are established.

\subsection{Layout}

The present paper is structured as follows: In the remainder of Section~1, we introduce our model problem as well as necessary notation and most notably, the Caffarelli-Silvestre extension problem. In Section~2, we formulate our main results: well-posedness of our formulation, the fully-discrete scheme using the diagonalization procedure together with the symmetric FEM-BEM coupling, and, finally, a best-approximation result. Section~3 provides the proofs for the well-posedness and the diagonalization procedure and, most notably, a Poincar\'e type estimate. Section~4 contains the proofs for the a-priori analysis of the fully discrete formulation using $hp$-finite elements in the extended variable, which builds upon the regularity and decay properties of \cite{part_one}. Finally, Section~5 presents some numerical examples that validate the proposed method. 

\subsection{Notations}
Throughout the text we use the symbol $a \lesssim b$ meaning that $a \leq Cb$ with a generic constant $C>0$ that is independent of any crucial quantities in the analysis. Moreover, we write $\simeq$ to indicate that both estimates $\lesssim$ and $\gtrsim$ hold.

We employ classical integer order Sobolev spaces $H^k(\Omega)$ on (bounded) Lipschitz domains $\Omega$ and the fractional Sobolev spaces $H^t(\R^d)$ for $t \in \R$ defined, e.g., via Fourier transformation. We also need Sobolev spaces on the boundary $\Gamma:=\partial \Omega$ of a bounded Lipschitz domain $\Omega \subset \R^d$, denoted by $H^{t}(\Gamma)$  with $t \in [-1,1]$. One way to properly define them is by using local charts, see \cite{sauter_schwab} for details.

\subsection{Assumptions on the model problem}
Let $d=2,3$. We consider \eqref{eq:stat_model_prob} and, for functions $u \in L^2(\R^d)$, define the self adjoint operator  $\LL^{\beta}$
 using spectral calculus
$$
\LL^{\beta} u:=\int_{\sigma(\LL)} z^\beta dE \; u,
$$
where $E$ is the spectral measure of $\LL$ and $\sigma(\LL)$ is the spectrum of $\LL$. Using standard techniques this definition can be extended to tempered distributions.
\bigskip

For the data, we assume $f \in L^2(\Omega)$ and  $\mathfrak{A} \in L^{\infty}(\R^d,\R^{d\times d})$ is  pointwise symmetric and positive definite in the sense that there exists $\mathfrak{A}_0 >0$ such that
\begin{align*}
(\mathfrak{A}(x) y,y)_2 \geq \mathfrak{A}_0 \norm{y}_2^2 \qquad \forall y \in \R^d.
\end{align*}
In order to avoid several additional difficulties due to decay conditions at infinity, we assume $s \geq\sigma_0>0$ for the case $d=2$. 
\bigskip

Additionally, in order to be able to apply FEM-BEM coupling techniques, we make the following  assumptions on the coefficients in the model problem: There exists a bounded Lipschitz domain $\Omega \subseteq \R^d$ such that
\begin{enumerate}
\item $\support f \subseteq \Omega$,
\item  $\mathfrak{A} \equiv I$ in $\R^d\setminus \overline{\Omega}$.
\end{enumerate}

This is not the most general setting where our techniques can be applied. For example,
   also a lowest order term could be included in the definition of $\LL$ in~\eqref{eq:stat_model_prob}; see also Remark~\ref{rem:more_general_setting}.

\subsection{Degenerate elliptic extension}
In the same way as in our previous work \cite{part_one}, we use a reformulation of the fractional PDE as the Dirichlet-to-Neumann mapping for a degenerate elliptic PDE in a half space in $\R^{d+1}$, the so called Caffarelli-Silvestre extension, \cite{caffarelli_silvestre,stinga_torrea}. 

We recall the definition of weighted Sobolev spaces used for the extension. For any bounded open subset
$D\subset \R^d\times \R$, we define $L^2(y^\alpha,D)$ as the space of square integrable functions with respect to the weight $y^\alpha$ and the Sobolev space   $H^1(y^\alpha,D) \subset L^2(y^\alpha,D)$ of functions with finite norm
$$
\norm{\UU}_{H^1(y^\alpha,D)}^2 := \int\int_{D}y^\alpha\Big( \big|\nabla \UU(x,y)\big|^2 
  +\big| \UU(x,y)\big|^2 \Big)\,dx\,dy.
$$ 
We also employ the spaces $L^2(y^\alpha,(0,\YY))$ and $H^1(y^\alpha,(0,\YY))$ for $\YY \in (0,\infty]$ defined in an analogous way by omitting the $x$-integration.

For  \emph{unbounded} sets $D$, we additionally use the weight 
\begin{align*}
\rho(x,y):=(1+\abs{x}^2+\abs{y}^2)^{1/2} \qquad (x,y) \in \R^d\times \R
\end{align*}
to take care of the behaviour at infinity. In this case, we define the 
space $H^1_{\rho}(y^\alpha,D)$ as the space of all square integrable functions $\UU$  (with respect to the weight function $y^\alpha \rho^{-2}$) such that the norm
\begin{align}\label{eq:normHHw}
  \norm{\UU}_{H^1_{\rho}(y^\alpha,D)}^2 :=
  \int\int_{D}{y^\alpha \Big( \big|\nabla \UU(x,y)\big|^2
  + \rho(x,y)^{-2} \big| \UU(x,y)\big|^2 \Big)\,dx\,dy}
\end{align}
is finite. 
Commonly used cases are $D=\R^d\times \R^+$ (full space), $D = \R^d \times (0,\YY)$ for $\YY>0$ (corresponding to truncation in $y$-direction), or $D=\omega \times (0,\YY)$ 
for $\omega \subset \R^d$ and $\YY >0$.

Moreover, we also employ spaces acting only in $x$. Using the weight 
\begin{align*}
\rho_x(x) := \rho(x,0),
\end{align*}
we introduce $L^2_{\rho_x}(\R^d)$ and $H_{\rho_x}^1(\R^d)$ as in \eqref{eq:normHHw} by omitting the $y$-integration.
\bigskip


For functions $\UU$ in $\HHwr$, one can give meaning to their trace at $y=0$, which we denote by $\operatorname*{tr}_0\UU$. In fact, by \cite[Lemma 3.8]{KarMel19} and  \cite[Lem.~3.1]{part_one} we have the trace estimates
\begin{align}
  \begin{split}
        \label{eq:trace}
        \abs{\trace \UU}_{H^\beta(\R^d)}  
        &\lesssim \norm{\nabla \UU}_{L^2(y^\alpha,\R^d\times\R^+)}  \\
            \|(1+|x|^2)^{-\beta/2} \trace\UU\|_{L^2(\R^d)}
            &\lesssim \norm{\nabla \UU}_{L^2(y^\alpha,\R^d\times\R^+)} \qquad \text{ if } d=3.
          \end{split}
\end{align}

Then, the extension problem reads as: find $\UU \in \HHwr$ such that
\begin{subequations}\label{eq:extension}
\begin{align}
 - \fdiv\big(y^\alpha \mathfrak{A}_x \nabla \UU\big) &= 0 \qquad \text{in $\R^{d} \times \R^+$},\\
  d^{-1}_{\beta}\partial_{\nu^\alpha} \UU +s \trace\UU&= f \qquad \text{in $\R^{d}$},
\end{align}
\end{subequations}
where $d_{\beta}:=2^{1-2\beta}\Gamma(1-\beta)/\Gamma(\beta)$, $\alpha:=1-2\beta \in (-1,1)$, $\partial_{\nu^\alpha} \UU(x) := -\lim_{y\rightarrow 0} y^\alpha \partial_y \UU(x,y)$, and $\mathfrak{A}_x = \begin{pmatrix} \mathfrak{A} & 0 \\ 0 & 1 \end{pmatrix} \in \R^{(d+1) \times (d+1)}$.
By \cite{stinga_torrea}, the solution to \eqref{eq:stat_model_prob} is then given by $u = \trace \,\UU$.
\bigskip

For the domain $\Omega$ with boundary $\Gamma := \partial \Omega$, we also introduce the usual  trace operators $\gamma_\Gamma^-$ (denoting the trace coming from the interior of $\Omega$) and $\gamma_\Gamma^+$ (denoting the trace coming from $\R^d\backslash\overline{\Omega}$) and correspondingly the normal derivative operators $\partial_{\nu,\Gamma}^\pm$ (see \cite{sauter_schwab} for details).
The normal vector $\nu$ is always assumed to face out of $\Omega$. 
With theses operators, the jumps across $\Gamma$ are defined as
\begin{align}
  \tracejump{u}&=\gamma^{-}_{\Gamma} u - \gamma^+_\Gamma u,
  \qquad
  \normaljump{u}=\partial^-_{\nu,\Gamma}u - \partial^+_{\nu,\Gamma}u.
\end{align}
We will apply these operators for functions in $H^1_{\rho}(y^{\alpha},\R^d \setminus \Gamma \times \R^+)$, where they
are to be understood pointwise with respect to $y$.
This is equivalent to taking the trace and normal derivative along the lateral boundary $\Gamma \times \R^+$.
\bigskip

\section{Main results}

\subsection{Variational formulation}

The weak formulation of \eqref{eq:extension} in $\HHwr$ reads as finding $\UU \in \HHwr$ such that
\begin{align}\label{eq:weakform}
  A(\UU,\VV)&:=
  \int_{0}^{\infty}{y^\alpha \int_{\R^d}\mathfrak{A}_x(x)\nabla \UU \cdot \nabla \VV \; dx dy} + s d_{\beta} \int_{\R^d}{\trace \UU \trace \VV\,dx} = d_\beta (f,\trace \VV)_{L^2(\R^d)}
  \end{align}
for all $\VV \in \HHwr$. Well-posedness of the continuous formulation follows from \cite[Prop~2.3]{part_one}.
\bigskip

In order to also include our discretization scheme, we work in a slightly expanded
variational form, inspired by~\cite{LS09,sayas09}.
In short, one can formulate an equivalent problem for the solution inside $\Omega$ and 
a function $\UU_\star$  on $\R^d$ defined in a modified Hilbert space. 

\begin{definition}
  Fix $\YY \in (0,\infty]$. 
  We consider the space
 \begin{align*}
   \HH_\YY:=\Big\{(\UU_\Omega, \UU_\star) \in H^1(y^\alpha, \Omega \times (0,\YY)) \times H_\rho^1(y^\alpha, \R^d \setminus \Gamma\times (0,\YY)): \;& \\
   \qquad\tracejump{\UU_{\star}} = \gamma^-_{\Gamma} \UU_\Omega, \;
  \gamma^-_{\Gamma} \UU_{\star} = 0,\; 
  s \, \trace\UU_{\star} \in L^2(\R^d)&\Big\}
 \end{align*}
 equipped with the norm 
  \begin{align*}
    \|\UU\|_{\HH_\YY}^2&:=\|(\UU_\Omega,\UU_\star)\|_{\HH_\YY}^2 \\
    &:= \|\UU_{\Omega}\|_{H^1(y^\alpha, \Omega\times (0,\YY))}^2
      + \|\UU_{\star}\|_{H^1_\rho(y^\alpha, \R^d \setminus \Gamma\times (0,\YY))}^2 
      + s \|{\trace{\UU_\Omega}}\|_{L^2(\Omega)}^2
      + s \|{\trace{\UU_\star}}\|_{L^2(\R^d)}^2.
  \end{align*}
  \end{definition}
  
 We note that, by definition, the additional condition of $\trace \UU_\star$ being in $L^2(\R^d)$ is only needed for $s\neq 0$ as in this case the norm in $\HH_\YY$ contains said $L^2$-term, which has to be finite.
 \bigskip

With $\UU = (\UU_\Omega,\UU_\star) \in  \HH_{\infty}$ and $\VV = (\VV_\Omega,\VV_\star) \in  \HH_{\infty}$, we define the bilinear form $B: \HH_{\infty} \times \HH_{\infty}\to \R$ as
  \begin{align}\label{eq:modified_BLF}
    B(\UU,\VV)
    :=&\int_0^{\infty}\int_{\Omega}{y^\alpha \mathfrak{A}_x(x)\nabla \UU_\Omega \cdot \nabla \VV_{\Omega} \;dx dy}
      + \int_0^{\infty}\int_{\R^d}{y^\alpha \nabla \UU_\star \cdot  \nabla \VV_{\star}\;dx dy}                   \nonumber \\
    &+ s d_{\beta} \int_{\Omega}{\trace \UU_\Omega \trace \VV_{\Omega}\,dx} + s d_{\beta} \int_{\R^d}{\trace \UU_\star \trace \VV_\star\,dx}.                          
  \end{align}
For $f \in L^2(\Omega)$, the weak formulation is given as the problem of finding $\UU \in \HH_{\infty}$ such that
  \begin{align}
    \label{eq:weak_problem_stat}
    B(\UU,\VV) &= d_{\beta} \int_{\Omega}{f \,\trace \VV_\Omega} \; dx \qquad \forall \; \VV = (\VV_\Omega,\VV_\star) \in \HH_{\infty}.
  \end{align}
Problems \eqref{eq:extension} and \eqref{eq:weak_problem_stat} are connected as follows: If $\UU_{\infty} =  (\UU_\Omega,\UU_\star)  \in \HH_{\infty}$ solves 
\eqref{eq:weak_problem_stat}, then the function $\UU:=\begin{cases}\UU_\Omega, \quad \text{ in }\Omega \\ \UU_\star, \quad \text{ in }\R^d\backslash\overline{\Omega}\end{cases}$ solves \eqref{eq:weakform}. 
\bigskip

In order to obtain a computable formulation, we start by cutting the problem from the infinite cylinder $\R^d \times \R^+$ to a finite
cylinder in the $y$-direction. To do so, we fix a parameter $\YY >0$ to be chosen later and
introduce the truncated bilinear forms
\begin{align*}
  {A}_{\Omega}^{\YY}(\UU,\VV)&:=
  \int_{0}^{\YY}{y^\alpha \int_{\Omega}\mathfrak{A}_x(x)\nabla \UU \cdot \nabla \VV \;dx dy} + s d_{\beta} \int_{\Omega}{\trace \UU \trace \VV\,dx}, \\
    {A}_{\R^d\backslash\Gamma}^{\YY}(\UU,\VV)&:=
  \int_{0}^{\YY}{y^\alpha \int_{\R^d\backslash\Gamma}{\nabla \UU \cdot \nabla \VV} \;dx dy} + s d_{\beta} \int_{\R^d\backslash\Gamma}{\trace \UU \trace \VV\,dx}.
\end{align*}

The ``big'' bilinear form is then given by
\begin{align*}
  B^{\YY}\big((\UU_\Omega,\UU_\star),(\VV_{\Omega},\VV_\star)\big)
  :={A}_{\Omega}^{\YY}(\UU_{\Omega},\VV_{\Omega}) + {A}_{\R^d \setminus \Gamma}^{\YY}(\UU_{\star},\VV_{\star}),
\end{align*}
and the cutoff problem reads as: Find
$\UU^\YY = (\UU^\YY_\Omega,\UU^\YY_\star) \in \HH_{\YY}$  such that
\begin{align}\label{eq:truncated_BLF_eq}
  B^{\YY}\big(\UU^\YY,\VV^\YY\big)
  =d_\beta\big(f,\trace{\VV^\YY_\Omega}\big)_{L^2(\R^d)} \quad 
  \text{for all } \VV^\YY=(\VV^\YY_\Omega,\VV^\YY_\star) \in \HH_{\YY}.
\end{align}

By the following theorem, we obtain well-posedness of the weak formulation of both variational formulations.

\begin{theorem}
  \label{prop:cont_well_posedness}
Assume either $d=3$ or $s>0$. 
Then, problem \eqref{eq:weak_problem_stat} has a unique solution $\UU \in \HH_{\infty}$ 
satisfying 
  \begin{align*}
    \norm{\UU}_{\HH_\infty} &\leq C \min(1, s^{-1}) \norm{f}_{L^2(\Omega)}.
  \end{align*} 
Fix $\YY \in (0,\infty)$. Then, the truncated problem \eqref{eq:truncated_BLF_eq} has a unique solution $\UU^{\YY}\in \HH_{\YY}$, for which the estimate
  \begin{align*}
    \norm{\UU^\YY}_{\HH_\YY} &\leq C \left(1+\frac{1}{\YY}\right) \min(1, s^{-1}) \norm{f}_{L^2(\Omega)}
  \end{align*} 
holds.
  Additionally, the bilinear forms in \eqref{eq:weak_problem_stat} and \eqref{eq:truncated_BLF_eq} are coercive.
\end{theorem}

The proof of the theorem is given in Section~\ref{sec:well-posedness} and,  in fact, reduces to the application of suitable Poincar\'e inequalities.


\subsection{The discrete scheme}

In this section, we describe our discrete scheme to approximate solutions to the truncated variational formulation \eqref{eq:truncated_BLF_eq}. The main idea is to employ a tensor product structure for the approximation by using the diagonalization procedure described in \cite{tensor_fem}.
\bigskip

Let $\VHy$ be an arbitrary finite dimensional subspace of $L^2(y^\alpha,(0,\YY))$ of dimension $N_y+1$.
Following the ideas of~\cite{tensor_fem},
we chose an orthonormal basis $(\varphi_{j})_{j=0}^{N_y}$ of $\VHy$  in $L^2(y^\alpha,(0,\YY))$
and generalized eigenvalues $\mu_j \geq 0$ satisfying
\begin{align}\label{eq:eigenvalue_problem}
  \int_{0}^{\mathcal{Y}}{y^\alpha  \varphi_i' \varphi_j' \,dy} + s \varphi_i(0)\varphi_j(0) 
  = \mu_j \int_{0}^{\mathcal{Y}}{y^\alpha \varphi_i \varphi_j\,dy}  = \mu_j\delta_{ij} \qquad \forall \; 0\leq i,j\leq N_y.
\end{align}
It is easy to see that for $s=0$, the constant function is an eigenfunction corresponding
  to the
  eigenvalue $\mu=0$. Moreover, the assumption $s>0$ for $d=2$ guarantees that there is not a zero eigenvalue. If zero is an eigenvalue (for $d=3$), we assume that the eigenvalues are ordered such that $\mu_0=0$. 
\bigskip

We now give a formal definition of our (semi-)discrete subspace of $\HH_\YY$, which has tensor product structure with respect to the variables $x,y$.
\begin{definition}
Let $\VHx \subset H^1(\Omega)$ and $\VHl \subset H^{-1/2}(\Gamma)$ be finite dimensional spaces and $\YY \in (0,\infty)$.  Additionally, assume that $1 \in \VHl$.  We introduce the closed subspace $\HH_{h,\YY} \subset \HH_\YY$ as 
  \begin{alignat}{6}
    \HH_{h,\YY}:=\operatorname*{cls}\Big\{ \UU_h=(\UU_\Omega,\UU_\star) \in \HH_\YY: \quad
    &\UU_\Omega(x,y) = \sum_{j=0}^{N_y}{ u_{j,\Omega}(x) \varphi_{j}(y)} \text{ with } u_{j,\Omega} \in \VHx,  \nonumber \\
    &  \UU_{\star}(x,y)=\sum_{j=0}^{N_y}{ u_{j,\star}(x) \varphi_{j}(y)} \text{ with } u_{j,\star} \in H^1_{\rho_x}(\R^d \setminus \Gamma), \nonumber\\
    &\tracejump{u_{j,\star}} = \gamma^-_{\Gamma} u_j,
                             \quad \gamma^-_{\Gamma} u_{j,\star} \in (\VHl)^\circ
    \Big\}.
  \end{alignat}
\end{definition}

Then, the semi-discrete problem reads as:
Find
$\UU_h^\YY = (\UU^\YY_\Omega,\UU^\YY_\star) \in \HH_{h,\YY}$  such that
\begin{align}\label{eq:truncated_BLF_eq_discrete}
  B^{\YY}\big(\UU_h^\YY,\VV_h^\YY\big)
  =d_\beta\big(f,\trace{\VV^\YY_\Omega}\big)_{L^2(\R^d)} \quad 
  \text{for all } \VV_h^\YY=(\VV^\YY_\Omega,\VV^\YY_\star) \in \HH_{h,\YY}.
\end{align} 

Using the orthogonal basis for the $y$-direction, we can actually diagonalize some of the bilinear forms and obtain
an equivalent sequence of scalar problems. 
In fact,   functions $(\UU_\Omega,\UU_\star) \in \HH_{h,\YY}$ solve \eqref{eq:truncated_BLF_eq_discrete}, if and only if
  they can be written as
  \begin{align}
    \label{eq:y_decompotistion}
  \UU_{\Omega}(x,y)=\sum_{j=0}^{N_y}{ u_{j,\Omega}(x) \varphi_{j}(y)}, \qquad \UU_{\star}(x,y)=\sum_{j=0}^{N_y}{ u_{j,\star}(x) \varphi_{j}(y)},
  \end{align}
  with $u_{j,\Omega} \in \VHx$, $u_{j,\star} \in H^1_{\rho_x}(\R^d \setminus \Gamma)$,
  where the functions $u_{j,\Omega}$, $u_{j,\star}$ solve 
  \begin{subequations}
    \label{eq:diagonalized_eqns1}
\begin{multline}
    \big(\mathfrak{A}\nabla u_{j,\Omega}, \nabla v\big)_{L^2(\Omega)}
    +  \big(\mu_j u_{j,\Omega}, v\big)_{L^2(\Omega)}
        - \big<\partial_{\nu,\Gamma}^-u_{j,\star},\gamma^-_{\Gamma} v\big>_{L^2(\Gamma)} \\
    = d_{\beta}\varphi_j(0)(f, v)_{L^2(\Omega)}
    \quad \forall v \in \VHx,
    \label{eq:diagonalized_eqns_inside1}
  \end{multline}
  and
  \begin{align}
    -\laplace u_{j,\star} + \mu_j u_{j,\star} &=0  \qquad \text{in $\R^d \setminus \Gamma$}
    \label{eq:diagonalized_eqns_outside1}\\
    \tracejump{u_{j,\star}}&= \gamma^-_{\Gamma} u_j,
                             \quad \gamma^-_{\Gamma} u_{j,\star} \in (\VHl)^\circ.
                             \label{eq:diagonalized_eqns_traces1}
  \end{align}
\end{subequations}
We refer to Lemma~\ref{lemma:y_decomposition} for a proof of this statement.
\bigskip

The equation for $u_{j,\star}$ is still posed on an unbounded domain.  We
will replace this with boundary integral equations. Therefore,
given $\mu \in \C$ with $\Re(\mu)\geq 0$, we introduce
\begin{align*}
  G(z;\mu):=\begin{cases}
    \frac{\ii}{4} H_0^{(1)}\left(\ii \mu \abs{z}\right), & \text{ for $d=2$,} \\
      \frac{e^{-\mu\abs{z}}}{4 \pi \abs{z}}, & \text{ for $d = 3$,}
    \end{cases}
                                               \quad
                                               \text{for $\mu \neq 0$ and}                                             
   \quad
  G(z;0):=\begin{cases}
    \frac{-1}{2\pi} \ln(\abs{z}), & \text{ for $d=2$,} \\
      \frac{1}{4 \pi \abs{z}}, & \text{ for $d = 3$,}
    \end{cases} 
\end{align*}
where $H_0^{(1)}$ denotes the first kind Hankel function of order $0$. The single-layer and double-layer potential 
are then defined as 
\begin{align*}
  \big(\widetilde{V}(\mu) \varphi \big)\left(x\right)&:=\int_{\Gamma}{G(x-z;\mu) \varphi(z) \;dz}, \qquad
  \big(\widetilde{K}(\mu) \psi \big)\left(x\right):=\int_{\Gamma}{\partial^-_{\nu,\Gamma}G(x-z;\mu) \psi(z) \;dz}, 
\end{align*}
and the corresponding boundary integral operators are given by
\begin{alignat}{4}
  V(\mu)&:=\gamma^\pm_{\Gamma} \widetilde{V}(\mu), &\quad  K(\mu)&:=\frac{1}{2}(\gamma^+_{\Gamma} \widetilde{V}(\mu) + \gamma^-_{\Gamma} \widetilde{V}(\mu)), \\
  K^t(\mu)&:=\frac{1}{2}(\partial_{\nu.\Gamma}^+ \widetilde{K}(\mu) + \partial_{\nu,\Gamma}^- \widetilde{K}(\mu)), &\quad  W(\mu)&:=- \partial_{\nu,\Gamma}^- \widetilde{K}(\mu).
\end{alignat}
We then have the following result, giving a computable approximation of~\eqref{eq:stat_model_prob} that
only relies on well-known operators.

\begin{theorem}
  \label{pro:diagonalized_couplings}
Let $\varphi_j,\mu_j$ be the generalized eigenfunctions and eigenvalues from \eqref{eq:eigenvalue_problem}. 
  For all $j=0,\dots, N_y$,
  let $(u_j,\lambda_j) \in \VHx\times \VHl$ solve
  \begin{subequations}
    \label{eq:diagonalized_coupling_eqs}
  \begin{align}
    \big(\mathfrak{A}\nabla u_{j}, \nabla v_h\big)_{L^2(\Omega)}
    +  \big(\mu_j u_{j}, v_h\big)_{L^2(\Omega)}  
    +
    \big<W(\mu_j) \gamma^-_{\Gamma} u_{j}
    &+ (-1/2+ K'(\mu_j))\lambda_j,
    \gamma^-_{\Gamma} v_h\big>_{L^2(\Gamma)}  \nonumber   \\
    &=d_\beta \varphi_j(0) \big(f,v_h\big)_{L^2(\Omega)}, \\
    \big<(1/2-K(\mu_j))\gamma^-_{\Gamma} u_j,\xi_h\big>_{L^2(\Gamma)}
    + \big<V(\mu_j) \lambda_j,\xi_h\big>_{L^2(\Gamma)}&=0
  \end{align}
\end{subequations}
for all $v_h \in \VHx$ and $\xi_h \in \VHl$.
Then,
\begin{align*}
  \UU_{\Omega}(x,y):=\sum_{j=0}^{N_y}{u_j(x) \varphi_j(y)},
  \qquad
  \UU_{\star}(x,y):=\sum_{j=0}^{N_y}{
  \Big(\widetilde{V}(\mu_j) \lambda_j (x)
  -\widetilde{K}(\mu_j)\gamma^-_{\Gamma} u_j(x)
  \Big) 
  \varphi_j(y)}
\end{align*}
solves~\eqref{eq:truncated_BLF_eq_discrete}. We thus have a computable representation of our
discrete approximation.
\end{theorem}

  The problems \eqref{eq:diagonalized_coupling_eqs} are standard FEM-BEM coupling problems
  for what is often called the modified Helmholtz or Yukawa equation. As such, existence and uniqueness of solutions $(u_j,\lambda_j) \in \VHx \times \VHl$ is well-known, see \cite[Sect. 7]{LS09}. Consequently, we also obtain  well-posedness of the semi-discrete formulation \eqref{eq:truncated_BLF_eq_discrete} as we have constructed a solution in $\HH_{h,\YY}$. Uniqueness follows from coercivity of the bilinear form.

\begin{corollary}
  \label{prop:discrete_well_posedness}
 Fix $\YY \in (0,\infty)$. 
  Let $\VHx \subseteq H^1(\Omega)$, $\VHl \subseteq H^{-1/2}(\Gamma)$,  
  $\VHy \subseteq H^1(y^\alpha,(0,\YY))$ be finite dimensional subspaces.
  Assume that $1 \in \VHl$, i.e., the space $\VHl$ contains the constant functions, and
  either $d=3$ or $s>0$.
Then, the truncated problem \eqref{eq:truncated_BLF_eq_discrete} has a unique solution $\UU_h^{\YY}\in \HH_{h,\YY}$.
\end{corollary}

  \begin{remark}
  Due to the construction in Theorem~\ref{pro:diagonalized_couplings}, we mention that our discrete approximation can
 very easily be computed with the use of existing FEM/BEM libraries. We refer to Section~\ref{sec:numerics} for a description of the implementation used in the numerical examples therein.
\end{remark}

\subsection{A-priori convergence estimates}
In the extended variable $y$, we employ a $hp$-FEM discretization. 
\bigskip

  Let $\YY >0$ and   $\TT_y$ be a geometric grid on $(0,\YY)$ with mesh grading factor $\sigma$,
  $L$-refinement layers towards $0$, and $M = \lfloor \ln(\YY)/\ln(\sigma) \rfloor$ levels of growth towards $\YY$. More precisely, we define the grid points as
  \begin{align} \label{eq:geometric_mesh}
    x_0:=0, \qquad  x_{\ell}:=\sigma^{L-\ell} \;\text{ for $\ell=0,\dots, L+M$,} 
    \qquad  x_{L+M+1}:=\YY.
  \end{align} 
  By
$$
\SS^{p,1}(\TT_y) := \{u \in C(0,\YY) \; :\; u|_{(x_\ell,x_{\ell+1})} \in P_p\; \forall \ell=0,\dots,L+M\}
$$ 
we denote the space of continuous, piecewise polynomials  of degree up to $p$.
\bigskip

The following proposition provides a best-approximation estimate for the $hp$-semi-discretization in $y$. We note that in contrast to \cite{tensor_fem}, which exploits a known closed form representation of the solution, we only have algebraic convergence of the truncated solution rather than exponential convergence. However, choosing the truncation parameter large enough, the $hp$-semi-discretization in $y$ allows to recuperate any algebraic convergence rates of the discretization in $x$.

\begin{theorem}[Best-Approximation] \label{prop:bestApproximation} Let $\YY \in (0,\infty)$. Let $\UU$ solve \eqref{eq:weak_problem_stat} and $\UU^\YY = (\UU^\YY_\Omega,\UU^\YY_\star)$ solve the cutoff problem~\eqref{eq:truncated_BLF_eq}.
Set  $\lambda:=\partial_\nu^+ \UU^{\YY}_\star$.
Let $\TT_y$ be a geometric grid on $(0,\YY)$.
  Let $\UU_h^\YY$ solve \eqref{eq:truncated_BLF_eq_discrete} with arbitrary finite dimensional subspaces $\VHx \subseteq H^1(\Omega)$, $\VHl \subseteq H^{-1/2}(\Gamma)$ and 
the choice $\VHy  := \SS^{p,1}(\TT_y)$. Let $\pi_\Omega : L^2(\Omega) \rightarrow \VHx$ be an arbitrary linear operator that is stable in $L^2(\Omega)$ and $H^1(\Omega)$.
  Then, for any $\lambda_h: \R_+ \to  \VHl$, there exist $\varepsilon >0$, $\kappa>0$ such that there holds
  \begin{align*}
    \|\UU - \UU_h^\YY\|_{\HH_\YY}^2
    &\lesssim
      \int_{0}^{\YY}{
      y^\alpha\Big( \|(I- \pi_\Omega) \UU^\YY_{\Omega}(y) \|_{H^1(\Omega)}^2
      +\|\lambda(y) -\lambda_h(y)\|_{H^{-1/2}(\Gamma)}^2 \Big) dy} \\
    &\qquad+ 
     \YY^{2\varepsilon} e^{-2\kappa p}+ \YY^{-\mu}\norm{f}_{L^2(\R^d)}^2 
  \end{align*}
  with   $    \mu:=
    \begin{cases}
      1+\abs{\alpha} & \text{ for } s>0 \\
       1+\alpha & \text{ for }  s = 0
    \end{cases}
$
and all constants independent of $\YY,p$.
\end{theorem}

\begin{remark}
 A possible choice for the spatial discretization would be $\VHx := \SS^{1,1}(\TT_x)$, i.e., continuous, piecewise linear polynomials on some (quasi-uniform) mesh $\TT_x$ of $\Omega$.
 For the operator $\pi_\Omega$ one could take the Scott-Zhang projection mapping onto $\SS^{1,1}(\TT_x)$, see \cite{ScottZhang}. In addition to the required $L^2(\Omega)$- and $H^1(\Omega)$-stabilities, the operator has first order approximation properties in $H^1(\Omega)$, provided the input function is sufficiently regular. 
\end{remark}

Using first-order approximation properties of the Scott-Zhang projection together with best-approximation of the BEM part (which converges of order $h^{3/2}$ assuming sufficient regularity, see \cite{sauter_schwab}) and correct choice of the cut-off parameter $\YY$ and polynomial degree $p$, the best-approximation estimate for the semi-discretization in Theorem~\ref{prop:bestApproximation} directly gives first order convergence in $h$.

\begin{corollary}\label{cor:convergenceRates}
  Let the assumptions of Theorem~\ref{prop:bestApproximation} hold.
  Assume $\mathfrak{A} \in C^{1}(\R^d,\R^{d\times d})$ and $f \in H^1(\Omega)$
  and assume $\Omega$ has piecewise smooth boundary.
  Choose $\VHx := \SS^{1,1}(\TT_x)$ with a quasi-uniform mesh $\TT_x$ of $\Omega$ of maximal mesh-width $h$ and take $\pi_\Omega$ to be the Scott-Zhang projection. Let $\VHl := S^{0,0}(\TT_\Gamma)$ be the space of piecewise constants on the trace mesh $\TT_\Gamma$ of $\TT_x$. Moreover, choose $p =-c_{\kappa,\mu,\varepsilon}\ln h$ with a sufficiently large constant $c_{\kappa,\mu,\varepsilon}$ depending only on $\kappa,\mu$ and $\varepsilon$, and $\YY \sim h^{-2/\mu}$. Then,
  \begin{align*}
    \|\UU - \UU_h^\YY\|_{\HH_\YY}
    &\leq C h. 
  \end{align*}
\end{corollary}

\begin{remark}
  \label{rem:more_general_setting}
We note that our main results are valid for more general fractional PDEs as well. Using the same techniques, one obtains the statements also for:
\begin{enumerate}
\item $s\in\C$  with $\Re(s)\geq 0$;
\item operators containing lower order terms, i.e., 
 \begin{align*}
  \LL u := -\fdiv\big( \mathfrak{A} \nabla u\big) + \mathfrak{c} u,
\end{align*}
where $\mathfrak{c}: \R^d \rightarrow \R$ with $\mathfrak{c} \geq 0$ is smooth and satisfies $\mathfrak{c} \equiv \mathfrak{c}_0 \in \R$
  in $\R^d\setminus \overline{\Omega}$.
\end{enumerate}
\end{remark}

\section{Well-posedness and FEM-BEM formulation}
\label{sec:well-posedness}
In this section, we provide the proofs of Theorem~\ref{prop:cont_well_posedness} and Theorem~\ref{pro:diagonalized_couplings}. 

\subsection{Poincar\'e inequalities}
We now show the well-posedness of our variational formulations. The main ingredient is a Poincar\'e type estimate, which uses the following compactness result.

\begin{lemma}
  \label{lemma:compactness_on_bounded_domains}
  Let $D\subseteq \R^{d} \times \R^+$ be a bounded Lipschitz domain.  
  Assume $u_n \rightharpoonup 0$ weakly in $H^1(y^\alpha,D)$ and
  $\norm{\nabla u_n}_{L^2(y^\alpha,D)}\to 0$.
  Then, $u_n \to 0$ in $L^2(y^\alpha,D)$.
  \end{lemma}
  \begin{proof}
    We can cover the Lipschitz domain $D$ by a finite number of Lipschitz domains
    $D_1,\dots,D_m$ which are starshaped with respect to a ball,
    see for example \cite[Sect. 1.1.9, Lemma 1]{maz11}.
    Thus, without loss of generality we may assume that $D$ is
    starshaped with respect to a ball.
    With $c_n:=\int_{D}{u_n}$ we compute
     \begin{align*}
       \norm{u_n}_{L^2(y^\alpha,D)}^2
       &=\norm{u_n-c_n}_{L^2(y^\alpha,D)}^2 + 2 (u_n, c_n)_{L^2(y^\alpha,D)} - \norm{c_n}_{L^2(y^\alpha,D)}^2 \\
       &\lesssim \norm{ \nabla u_n}^2_{L^2(y^\alpha,D)} + 2\abs{(u_n, c_n)_{L^2(y^\alpha,D)}} \\
       &\leq \norm{ \nabla u_n}^2_{L^2(y^\alpha,D)} + 2\abs{c_n}\abs{(u_n,1)_{L^2(y^\alpha,D)}} \to 0,
     \end{align*}
     where we used the Poincar\'e estimate of~\cite[Corollary~4.4]{NOS15} and the assumed weak convergence.
  \end{proof}
  

 \begin{lemma}
    \label{lemma:my_poincare}
    Fix $\YY \in (0,\infty]$. Let  $\UU \in \WHwr$  with $\int_{\Gamma}{\tracejump{\UU}} ds_x=0$ for almost every $y \in (0,\YY)$.
 \begin{enumerate}
 \item Let $0\leq \mu \leq 2$ and $\YY = \infty$.  There holds 
  \begin{align}
    \label{eq:my_poincare1}
\int_{0}^{\infty}\int_{\R^d}{y^{\alpha}\rho^{\mu-2} |\UU|^2 \,dx} dy
    &\leq C
\int_{0}^{\infty}\int_{\R^d\backslash\Gamma}{y^{\alpha}\rho^{\mu} |\nabla \UU|^2 \,dx} dy 
  \end{align}
  provided the right-hand side is finite. 
 
 \item  Let $\YY\in (0,\infty)$. There exists $\mu_0 > 0$ such that for all $\mu \in [0,\mu_0)$ there  holds
   \begin{align}
    \label{eq:my_poincare}
\int_{0}^{\YY}\int_{\R^d}{y^{\alpha}\rho^{\mu-2}|\UU|^2 \,dx} dy
    &\leq C\left(
 \int_{0}^{\YY}\int_{\R^d\backslash\Gamma}{y^{\alpha}\rho^{\mu}|\nabla\UU|^2 \,dx} dy + |3-d| \|\trace\UU\|_{L^2(\R^d)}^2\right)
  \end{align}
  provided the right-hand side is finite.
 \end{enumerate}
\end{lemma}

\begin{proof}
The estimates follow from techniques employed in \cite[Theorem 3.3]{AGG94} using a proof by contradiction.
In the first step, we show \eqref{eq:my_poincare1} (which essentially is covered by \cite[Theorem 3.3]{AGG94},  we only account for the additional weight $y^\alpha$) and \eqref{eq:my_poincare} for functions vanishing inside a ball containing the origin. Finally, using a compactness argument this assumption is removed in the second step.


\textbf{Step 1:}
  First, assume that $\UU \equiv 0$ on a sufficiently large (half) ball $B_R(0) \subset \R^{d+1}$ and has compact support. 

  We employ spherical coordinates in $\R^{d}\times\R^+$, chosen such that $y=r \cos(\varphi)$ and
   collect the remaining $d-1$ angles into $\hat{\varphi}$. 
    Using $\rho^{\mu-2} = (1+\abs{x}^2+y^2)^{-(\mu-2)/2} < r^{\mu-2}$ for $\mu \leq 2$, we calculate
   \begin{align*}
     \int_{0}^{\infty}{y^\alpha \int_{\R^d}{
     \rho^{\mu-2} |\UU(x,y)|^2\,dxdy}}
     &\lesssim\int_{\hat{\varphi}} \int_{\varphi=-\pi/2}^{\pi/2}\int_{R}^{\infty}{
       r^{d+\alpha+\mu-2} \cos(\varphi)^\alpha |\UU(x,y)|^2 |J(\varphi,\hat{\varphi})| \,dr\,d\varphi d\hat{\varphi}},
   \end{align*}
where we denoted by $J(\varphi,\hat{\varphi})$ the angular components of the Jacobian in the transformation theorem. Integration by parts in $r$ and using the assumed support properties of $\UU$ gives 
   \begin{multline*}
     \int_{\hat{\varphi}} \int_{-\pi/2}^{\pi/2}\int_{R}^{\infty}{
     r^{d+\alpha+\mu-2} \cos(\varphi)^\alpha |\UU(x,y)|^2 |J(\varphi,\hat{\varphi})| \,dr\,d\varphi d\hat{\varphi}}\\
     \begin{aligned}[t]&\lesssim
       \int_{\hat{\varphi}} \int_{-\pi/2}^{\pi/2}\int_{R}^{\infty}{
         r^{d+\alpha+\mu-1} \cos(\varphi)^\alpha |\UU(x,y)||\nabla \UU(x,y)| |J(\varphi,\hat{\varphi})| \,dr\,d\varphi d\hat{\varphi}}
       \\
       &\lesssim
       \Big(\int_{\hat{\varphi}} \int_{-\pi/2}^{\pi/2}\int_{R}^{\infty}{
         r^{d+\alpha+\mu-2} \cos(\varphi)^\alpha \UU(x,y)^2|J(\varphi,\hat{\varphi})| \,dr\,d\varphi d\hat{\varphi}}\Big)^{1/2} \\
       &\quad\times 
       \Big(\int_{\hat{\varphi}} \int_{-\pi/2}^{\pi/2}\int_{R}^{\infty}{
         r^{d+\alpha+\mu} \cos(\varphi)^\alpha \abs{\nabla \UU(x,y)}^2|J(\varphi,\hat{\varphi})| \,dr\,d\varphi d\hat{\varphi}}\Big)^{1/2}.
     \end{aligned}
   \end{multline*}
Transforming back to $(x,y)$-variables and using $r^\mu \leq \rho^\mu$ for $\mu \geq 0$ this gives the desired bound.
    By density, we can remove the requirement of compact support of $\UU$.

\textbf{Step 2:} 
  In order to get rid of the requirement that $\UU$ vanishes on the ball $B_{R}(0)$, we
  use a compactness argument. To keep the notation succinct we set $\mu=0$ in the following, the general case $\mu>0$ can be done with the exact same arguments.
   Assume that \eqref{eq:my_poincare1} does not hold, i.e., there exists
  a sequence $\UU_n \in \WHwrInf$ such that
  \begin{align*}
\norm{\UU_n}_{\WHwrInf}=1, \qquad
    \int_{0}^{\infty}{y^{\alpha} \|\nabla \UU_n(y)\|^2_{L^2(\R^d\backslash\Gamma)} \,dy}
              \leq \frac{1}{n}.
  \end{align*}
  Since $\UU_n$ is a bounded sequence in the Hilbert space $\WHwrInf$, there exists a
  weakly convergent subsequence (also denoted by $\UU_n$) and we denote the weak limit by $\UU$.

  Since the seminorm is lower semicontinuous, we get $\abs{\UU}_{\WHwrInf}=0$.
  A simple calculation (using polar coordinates, similar to the estimate above) shows that -- as we are in half-space in $\R^{d+1}$ with $d+1>2$ and $\int_{\Gamma}{\tracejump{\UU}ds_x}=0$
    -- the space $\WHwr$ does not contain piecewise constant functions except for $0$,  which means that $\UU=0$.
  
  We now show strong convergence of the sequence to $\UU=0$. To that end, fix a ball $B_R:=B_R(0) \subset \R^{d+1}$ with sufficiently large $R$ such that $\Omega \times \{0\} \subset B_{R}$ and
  consider a smooth cutoff function $\psi: \R^{d+1} \to \R$ such that $\psi \equiv 1$ on
  $B_R$ and $\psi \equiv 0$ on $B_{2R}$.
  We thus decompose $\UU_n$ as
  $$
  \UU_n=\psi \UU_n + (1-\psi)\UU_n=:\UU_n^1 + \UU_n^2.
  $$

  From the compactness result of Lemma~\ref{lemma:compactness_on_bounded_domains}
  applied to $\Omega \times \R^+\cap B_{\widetilde R}$ and $B_{\widetilde R} \setminus \overline{\Omega} \times \R^+$ separately,
  we deduce that
  $\UU_{n} \to \UU$ in $L^2(y^\alpha,B_{\widetilde R})$ and thus $\UU_n^1 \to \psi \UU = 0$ in $L^2(y^\alpha,B_{\widetilde R})$ on all bounded half balls $B_{\widetilde R}$ with sufficiently large $\widetilde R$.
  
  Since $\UU_n^2$ vanishes on $B_R$, we can apply step 1 of the proof
  to determine:
  \begin{align*}
    \norm{\UU_n^2}_{\WHwrInf}
    &\lesssim \abs{\UU_n^2}_{\WHwrInf} \\
      &\lesssim \abs{(1-\psi)\UU_n}_{H^1_\rho(y^\alpha,B_{2R}(0) \backslash \Gamma \times (0,\infty))} + \abs{\UU_n}_{H^1_\rho(y^\alpha,B_{2R}(0)^c\backslash \Gamma \times (0,\infty))} \\
       &\lesssim \norm{\UU_n}_{L^2(y^\alpha,B_{2R}(0))} + \abs{\UU_n}_{H_\rho^1(y^\alpha,\R^d\backslash\Gamma \times (0,\infty))} 
      \to 0.
  \end{align*}
  Overall, we get that $\UU_n \to 0$ in $\WHwrInf$, which is a contradiction to the assumption $\norm{\UU_n}_{\WHwrInf}=1$ for all
  $n \in \N$. 

\textbf{Step 3:} Estimate \eqref{eq:my_poincare} for the case $d=3$ follows directly from multiplying a full-space Poincar\'e-inequality (see for example~\cite[Theorem 3.3]{AGG94} for $\mu = 0$ and a similar calculation to step~1 for $0<\mu\leq 2$ with polar coordinates only in $x$) applied only in $x$ with $y^\alpha$ and integrating over $(0,\YY)$.

\textbf{Step 4:} It remains to show \eqref{eq:my_poincare} for $d=2$, which was also shown in \cite[Lem.~3.2]{part_one}. 
We write $\UU(x,y) = \UU(x,0) + \int_0^y \partial_y \UU(x,\tau) \; d\tau$, which gives
\begin{align*}
\int_0^\YY\int_{\R^d}y^\alpha \rho^{\mu-2} |\UU|^2 \; dxdy \lesssim \int_0^\YY\int_{\R^d}y^\alpha \rho^{\mu-2} |\UU(x,0)|^2 + y^\alpha \rho^{\mu-2}\Big( \int_0^y \partial_y \UU(x,\tau) \; d\tau \Big)^2 \; dxdy.
\end{align*} 
Since $\int_0^\YY y^\alpha \rho^{\mu-2} dy \lesssim 1$ for sufficiently small $\mu < \mu_0$ with $\mu_0$ depending only on $\alpha$, the first term on the left-hand side can be bounded by $C \norm{\trace \UU}_{L^2(\R^d)}^2$. For the second term, we employ a weighted Hardy-inequality, see e.g. \cite{muckenhoupt72}, to obtain
\begin{align*}
 \int_0^\YY\int_{\R^d}y^\alpha \rho^{\mu-2}\Big( \int_0^y \partial_y \UU(x,\tau) \; d\tau \Big)^2 \; dxdy \lesssim \int_{\R^d} \int_0^\YY y^\alpha \rho^\mu |\partial_y \UU|^2 \; dy dx,
\end{align*}
which shows the claimed inequality.
\end{proof}

We can now look at the well-posedness of our discrete problem.

\begin{proof}[Proof of Theorem~\ref{prop:cont_well_posedness}]
Let $\YY\in (0,\infty]$ and $(\UU_\Omega^\YY,\UU_\star^\YY) \in \HH_\YY$.
  On the interior domain $\Omega$, we integrate a standard Poincar\'e-like estimate
  to obtain
\begin{align}\label{eq:poincInt}
\int_{0}^{\YY}{y^\alpha \int_{\Omega}{\rho^{-2}|\UU_\star^\YY|^2dx}dy} \leq  \int_{0}^{\YY}{y^\alpha \int_{\Omega}{|\UU_\star^\YY|^2dx}dy}
  \lesssim \int_{0}^{\YY}{y^\alpha \int_{\Omega}{|\nabla_x \UU_\star^\YY|^2dx}dy}.
\end{align}
  By the conditions on $\gamma^- \UU_{\star}^\YY$ and $\tracejump{\UU_\star^\YY}$,
  we observe that the function
  $\UU^\YY:=\begin{cases}\UU_\Omega^\YY, \quad \text{ in }\Omega \\ \UU_\star^\YY, \quad \text{ in }\R^d\backslash\overline{\Omega}\end{cases}$ has a jump across $\partial \Omega$ with
  vanishing integral mean.
  Applying Lemma~\ref{lemma:my_poincare} to $\UU^\YY$, we get with \eqref{eq:poincInt} that
\begin{align*}
  B^\YY(\UU_h^\YY,\UU_h^\YY)  &\gtrsim
  \int_0^\YY \int_{\Omega}y^\alpha \abs{\nabla\UU_\Omega^\YY}^2 dx dy +   \int_0^\YY \int_{\R^d\backslash\Gamma}y^\alpha \abs{\nabla\UU_\star^\YY}^2 dx dy \\ 
& \quad +  s \|{\trace{\UU_\Omega^\YY}}\|_{L^2(\Omega)}^2
      + s \|{\trace{\UU_\star^\YY}}\|_{L^2(\R^d)}^2 \\
&=
  \int_0^\YY \int_{\R^d \backslash\Gamma}y^\alpha \abs{\nabla\UU^\YY}^2 dx dy +   \int_0^\YY \int_{\Omega}y^\alpha \abs{\nabla\UU_\star^\YY}^2 dx dy \\ 
& \quad +  s \|{\trace{\UU_\Omega^\YY}}\|_{L^2(\Omega)}^2
      + s \|{\trace{\UU_\star^\YY}}\|_{L^2(\R^d)}^2 \\
      & \gtrsim
   \norm{\UU^\YY}_{H^1_\rho(y^\alpha,\R^d\backslash \Gamma \times (0,\YY))}^2 +  \norm{\UU_\star^\YY}_{H^1_\rho(y^\alpha,\Omega \times (0,\YY))}^2 +
 s \|{\trace{\UU_\Omega^\YY}}\|_{L^2(\Omega)}^2
      + s \|{\trace{\UU_\star^\YY}}\|_{L^2(\R^d)}^2  \\
  &= \|(\UU_\Omega^\YY, \UU_\star^\YY)\|_{\HH_\YY}^2,
\end{align*}
which shows coercivity.

 In order to bound the right-hand side in \eqref{eq:truncated_BLF_eq}, we can directly use the definition of the $\HH_{\YY}$-norm together with $\operatorname*{supp} f \subset \Omega$ for $s>0$ to obtain
  \begin{align*}
    \int_{\R^d} f \trace{\VV^\YY_\Omega} \; dx
    &\leq s^{-1}\norm{f}_{L^2(\Omega)} s\norm{\trace{\VV_\Omega^\YY}}_{L^2(\R^d)}
      \leq s^{-1}\norm{f}_{L^2(\Omega)} \norm{\VV_\Omega^\YY}_{\HH_{\YY}}.
  \end{align*}
  For $\YY = \infty$ and $s=0$, which implies $d=3$ by assumption, the trace estimate \eqref{eq:trace} gives 
  \begin{align*}
    \int_{\R^d} f \trace{\VV_\Omega} \; dx
    &\leq \|{\rho(x,0)^{\beta}f}\|_{L^2(\Omega)} \|{\rho(x,0)^{-\beta}\trace{\VV_\Omega}}\|_{L^2(\R^d)}
\lesssim \norm{f}_{L^2(\Omega)} \norm{\nabla \VV_\Omega}_{L^2(y^\alpha,\R^d \times \R^+)} \\
     & \leq \norm{f}_{L^2(\Omega)} \norm{\VV}_{\HH_\infty}.
  \end{align*}
For the case $\YY < \infty$ and $s=0$, we use a cut-off function $\chi$ satisfying $\chi \equiv 1$ on $(0,\YY/2)$, $\operatorname*{supp}\chi \subset (0,\YY)$ and $\norm{\nabla \chi}_{L^\infty(\R^+)}\lesssim \YY^{-1}$. As $\Omega$ is bounded, this gives with the trace estimate \cite[Lem.~3.7]{KarMel19}
  \begin{align*}
    \int_{\R^d} f \trace{\VV^\YY_\Omega} \; dx
    &\leq \norm{f}_{L^2(\Omega)} \norm{\trace{(\chi\VV^\YY_\Omega)}}_{L^2(\Omega)} \\
&\lesssim \norm{f}_{L^2(\Omega)} \left(\norm{\chi \VV^\YY_\Omega}_{L^2(y^\alpha,\Omega\times (0,\YY))} +\norm{\nabla(\chi \VV^\YY_\Omega)}_{L^2(y^\alpha,\Omega \times (0,\YY))} \right) \\
&\lesssim  \norm{f}_{L^2(\Omega)} \left(\Big(1+\frac{1}{\YY} \Big)\norm{\VV^\YY_\Omega}_{L^2(y^\alpha,\Omega\times (0,\YY))} +\norm{\nabla \VV^\YY_\Omega}_{L^2(y^\alpha,\Omega \times (0,\YY))} \right) \\
     & \leq C \left(1+\frac{1}{\YY}\right)\norm{f}_{L^2(\Omega)} \norm{\VV^\YY_\Omega}_{\HH_{\YY}},
  \end{align*}
  which finishes the proof.
\end{proof}

\subsection{Diagonalization}
We now apply the diagonalization procedure of \cite{tensor_fem} to show that solutions of \eqref{eq:truncated_BLF_eq_discrete} can be written as in \eqref{eq:y_decompotistion}.
We recall that $(\varphi_j)_{j=0}^{N_y}$ is the orthonormal basis of eigenfunctions
  from~\eqref{eq:eigenvalue_problem} with corresponding eigenvalues $\mu_j$.

\begin{lemma}\label{lemma:y_decomposition} 
  Functions $(\UU_\Omega^\YY,\UU_\star^\YY) \in \HH_{h,\YY}$ solve \eqref{eq:truncated_BLF_eq_discrete}, if and only if
  they can be written as
  \begin{align*}
  \UU_{\bullet}^\YY(x,y)=\sum_{j=0}^{N_y}{ u_{j,\bullet}(x) \varphi_{j}(y)},
  \end{align*}
  where $\bullet \in \{\Omega,\star\}$ and
  \begin{align*}
    u_{j,\Omega} \in \VHx, \qquad 
    u_{j,\star} \in H^1_{\rho_x}(\R^d \setminus \Gamma) \quad \forall j\geq 0
  \end{align*}
  such that for all $v \in \VHx$
\begin{subequations}
  \label{eq:diagonalized_eqns}
  \begin{align}
    \big(\mathfrak{A}\nabla u_{j,\Omega}, \nabla v\big)_{L^2(\Omega)}
    +  \big(\mu_j u_{j,\Omega}, v\big)_{L^2(\Omega)}
    &+ \big<\normaljump{u_{j,\star}},\gamma^-_{\Gamma} v\big>_{L^2(\Gamma)}
    = d_{\beta}\varphi_j(0)(f, v)_{L^2(\Omega)}
    \label{eq:diagonalized_eqns_inside}
    \\
    -\laplace u_{j,\star} + \mu_j u_{j,\star} &=0  \qquad \text{in $\R^d \setminus \Gamma$},
    \label{eq:diagonalized_eqns_outside}\\
    \tracejump{u_{j,\star}}&= \gamma^-_{\Gamma} u_j,
                             \quad \gamma^-_{\Gamma} u_{j,\star} \in (\VHl)^\circ.
                             \label{eq:diagonalized_eqns_traces}
  \end{align}
  \end{subequations}
\end{lemma}
\begin{proof}
  At first, we show unique solvability of \eqref{eq:diagonalized_eqns}. For that, we consider the weak formulation of \eqref{eq:diagonalized_eqns} given by
    \begin{multline}
      \label{eq:weak_form_of_diagonalized_problems}
  (\mathfrak{A}\nabla u_{j,\Omega}, \nabla v_{j,\Omega})_{L^2(\Omega)}
  + (\mu_j u_{j,\Omega}, v_{j,\Omega})_{L^2(\Omega)}+(\nabla u_{j,\star}, \nabla v_{j,\star})_{L^2(\R^d\backslash\Gamma)}+
  \mu_j (u_{j,\star},v_{j,\star})_{L^2(\R^d)}\\
  =d_{\beta} \varphi_j(0)(f,v_{j,\Omega})_{L^2(\R^d)}.
\end{multline}
The equivalence between the weak form and the strong form follows from standard arguments, and we refer to \cite[Sect.7]{LS09}.
Coercivity of the weak formulation in $H^1(\Omega) \times H^1(\R^d\backslash\Gamma)$ is clear for $\mu_j>0$ as $\mathfrak{A}$ is positive definite.
For $\mu_j=0$, one can employ Poincar\'e estimates on $\Omega$ and $\R^d$ (with weights) to obtain coercivity in $H^1(\Omega) \times H^1_{\rho_x}(\R^d\backslash\Gamma)$. Therefore, for each $j$, a unique solution $(u_{j,\Omega},u_{j,\star}) \in \VHx \times H^1_{\rho_x}(\R^d \backslash\Gamma) \subset H^1(\Omega) \times H^1_{\rho_x}(\R^d\backslash\Gamma)$ exists.

  We now show that if the $u_{j,\bullet}$ solve \eqref{eq:diagonalized_eqns}
    then  $\UU_h^\YY:=(\UU_\Omega^\YY,\UU_\star^\YY)$ with $\UU_{\bullet}^\YY:=\sum_{j=0}^{N_y}{ u_{j,\bullet} \varphi_{j}}$ solves~\eqref{eq:truncated_BLF_eq_discrete}. By
    construction we have $\UU_h^{\YY} \in \HH_{h,\YY}$.
    We next look at the weak formulation of~\eqref{eq:truncated_BLF_eq_discrete}. First we focus on the $\star$-contribution. Taking $\VV_\star^\YY =  v_{j,\star}(x) \varphi_j(y)$ with arbitrary $v_{j,\star} \in \VHx$ as test function, we compute
  \begin{align*}
  A^\YY_{\R^d\setminus \Gamma}(\UU_\star^\YY,\VV_\star^\YY)
  &= \sum_{\ell=0}^{N_y}
        \int_{\R^d}{u_{\ell,\star}(x) v_{j,\star}(x) dx}\int_{0}^{\YY}{y^\alpha  \varphi'_\ell(y) \varphi'_j(y)\,dy} \\
  &\qquad+ 
        \int_{\R^d\backslash\Gamma}{\nabla u_{\ell,\star}(x) \nabla v_{j,\star}(x) dx}\int_{0}^{\YY}{y^\alpha  \varphi_\ell(y) \varphi_j(y)\,dy}\\
  &\qquad+ 
        \int_{\R^d}{u_{\ell,\star}(x) v_{j,\star}(x) \; dx} \cdot s\varphi_\ell(0) \varphi_j(0)\\
  &=  
    \mu_j \int_{\R^d}{u_{j,\star}(x) v_{j,\star}(x) dx} +
    \int_{\R^d\backslash\Gamma}{\nabla u_{j,\star}(x) \nabla v_{j,\star}(x) dx}.
\end{align*}
For the interior contribution, the same diagonalization  procedure gives
for $\VV^{\YY}_{\Omega}:=v_{j,\Omega}(x)\varphi_j(y)$
\begin{align*}
  A^\YY_{\Omega}(\UU_{\Omega}^\YY, \VV_{\Omega}^\YY)
  &=
  (\mathfrak{A}\nabla u_{j,\Omega}, \nabla v_{j,\Omega})_{L^2(\Omega)}
  +  (\mu_j u_{j,\Omega}, v_{j,\Omega})_{L^2(\Omega)}.
\end{align*}
Summing up, and using the weak form \eqref{eq:weak_form_of_diagonalized_problems} we get that
$$
B^{\YY}(\UU_h^\YY,\VV_{h}^{\YY})= d_\beta (f,\trace{\VV_{\Omega}^{\YY}})_{L^2(\R^d)}
\qquad \text{for all } \VV_{h}^{\YY} = (\VV_{\Omega}^{\YY},\VV_{\star}^{\YY}) = \Big(\sum_{j=0}^{N_y}{v_{j,\Omega} \varphi_j },\sum_{j=0}^{N_y}{v_{j,\star} \varphi_j }\Big).
$$
By density we can extend this equality to all test functions $\VV_{h}^{\YY}$ in the space $\HH_{h,\YY}$ and get~\eqref{eq:truncated_BLF_eq_discrete}. Since the bilinear form
$B^{\YY}(\cdot,\cdot)$ is coercive we get that the function
$\UU^{\YY}$ thus constructed is the only solution to~\eqref{eq:truncated_BLF_eq_discrete},
which establishes the stated equivalence.
\end{proof}

\begin{proof}[Proof of Theorem~\ref{pro:diagonalized_couplings}]
The statement follows from Lemma~\ref{lemma:y_decomposition} and \cite[Section 7]{LS09}, as defining 
  $u_{j,\star}(x):=\widetilde{V}(\mu_j) \lambda_j (x)
  -\widetilde{K}(\mu_j)\gamma^-_{\Gamma} u_{j,\Omega}(x)$ and plugging that into \eqref{eq:diagonalized_eqns} gives the stated equations using classical properties of the layer potentials.  

  By definition and decay of the layer potentials, we have that $u_{j,\star} \in H^1(\R^d\backslash\Gamma)$ for all $j \in \N_0$ such that $\mu_j\neq 0$, which gives  $u_{j,\star} \in H^1_{\rho_x}(\R^d\backslash\Gamma)$ as well.
    If $s >0$, no zero eigenvalue is possible. This matches with the
      the requirement $s \,\trace \UU_\star \in L^2(\R^d)$ in the definition of the space $\HH_\YY$.
    The case $s=0$ is only allowed for $d=3$. Here, $\trace \UU_\star$
    is not required to be in $L^2(\R^d)$, which would not hold. However, in this case the decay property of the layer potentials for the Poisson equation, see e.g. \cite{sauter_schwab}, give  $u_{j,\star} \in H^1_{\rho_x}(\R^d\backslash\Gamma)$. In short, we have that our constructed solution is in the semi-discrete space $\HH_{h,\YY}$.
 
   Notably,
  \eqref{eq:diagonalized_eqns} is just the ``non-standard transmission problem''
  corresponding to the standard symmetric FEM-BEM coupling given by Theorem~\ref{pro:diagonalized_couplings}.
\end{proof}

\section{Error analysis}
The key to the error analysis are the decay and regularity properties shown in \cite{part_one}. In order to make the present paper more accessible, we summarize the key results of \cite{part_one} in the following.

\subsection{Decay and regularity}
The solution to the truncated problem is in fact a weak solution to a Neumann problem. 
Thus, in this section, we consider solutions $\UU^\YY$ to the following truncated problem:
\begin{subequations}    \label{eq:truncated_problem_pw}
  \begin{align}
    -\fdiv\big(y^\alpha \mathfrak{A}_x \nabla {\UU}^\YY\big) &= 0  &&\text{in $\R^{d} \times (0,\YY)$},\\
    d_{\beta}^{-1}\partial_{\nu^\alpha} {\UU}^\YY + s \trace{\UU}^\YY&=f &&\text{on $\R^{d} \times \{0\}$} , \\
    \partial_{y} {\UU}^\YY&=0 &&\text{on $\R^{d} \times \{\YY\}$}.
  \end{align}
\end{subequations}
 
Then, the truncation error can be controlled via the following proposition.
\begin{proposition}[Decay in $y$, {\cite[Prop.~2.5]{part_one}}]
  \label{prop:truncation_error}
  Fix $\YY >0$. Let $\UU$ solve~\eqref{eq:extension} and $\UU^\YY$ solve~\eqref{eq:truncated_problem_pw}.
   Let $\mu$ be given by   $ \mu:=
    \begin{cases}
      1+\abs{\alpha} & s>0 \\
       1+\alpha & s = 0
    \end{cases}.
$
  Then, the following estimate holds:
  \begin{align*}
    \|\UU^{\YY} - \UU\|^2_{\HHwry}
    +s\|\trace(\UU^{\YY} - \UU)\|^2_{L^2(\R^d)}
    \lesssim \YY^{-\mu}   
    \norm{f}^2_{L^2(\Omega)}.
  \end{align*}
\end{proposition}

The goal in the following is to employ $hp$-FEM in the extended variable $y$. Therefore, weighted analytic regularity estimates are the key to the a-priori analysis.  

\begin{proposition}[Regularity in $y$, {\cite[Prop.~2.6]{part_one}}]
  \label{prop:regularity}
  Fix $\YY \in (0,\infty]$ and let
  $\ell \in \N$. Let $\UU$ solve \eqref{eq:truncated_problem_pw}. Then,  
  there exist constants $C,K>0$
  and $\varepsilon \in (0,1)$ such that the following estimate holds:
  \begin{align*}
    \big\|{y^{\ell-\varepsilon}\nabla\partial^{\ell}_y \UU}\big\|_{L^2(y^\alpha,\R^d\times(0,\YY))}
    \leq CK^{\ell} \ell! \norm{f}_{L^2(\Omega)}.
  \end{align*}
  All constants are independent of $\ell, \YY,$ and $\UU$.
\end{proposition}

Denoting by   $L^2(y^\alpha,(0,\YY);X)$  the Bochner spaces of square integrable functions (with respect to the weight $y^\alpha$) and values in the Banach space $X$, the regularity results of the previous Proposition can be captured by the solution being in some countably normed space. For constants $C,K>0$, we introduce
\begin{align*}
  \mathcal{B}^1_{\varepsilon,0}(C,K;\YY,X) := \Big\{ \VV \in C^\infty((0,\YY);X) : & 
\norm{\VV}_{L^2(y^\alpha,(0,\YY);X)} < C, \\
&\norm{y^{\ell+1-\varepsilon}\VV^{(\ell+1)}}_{L^2(y^\alpha,(0,\YY);X)} < C K^{\ell+1} (\ell+1)! \; \forall \ell \in \N_0
\Big\}.
\end{align*}

\begin{corollary}\label{cor:regularity_y}
Fix $\YY \in (0,\infty]$ and let $\UU^\YY$ solve \eqref{eq:truncated_problem_pw}.
  Then, there are constants $C,K>0$ such that there holds 
  \begin{align}
    \partial_{y}\UU^\YY \in \mathcal{B}^1_{\varepsilon,0}(C,K;\YY,L^2(\R^d))
    \qquad \text{and} \qquad
    \nabla_{x}\UU^\YY \in \mathcal{B}^1_{\varepsilon,0}(C,K;\YY,L^2(\R^d)).
  \end{align}
\end{corollary}

\subsection{Fully discrete analysis}
In order to derive error bounds, we employ the reformulation in \eqref{eq:modified_BLF} together with the already established decay bounds for the truncation in $\YY$. 

We will need two quasi-interpolation operators -- one for the $x$-variables and one for the $y$-direction.
Their construction and properties are the subject of the next two lemmas.

\begin{lemma}[Interpolation in $x$]
  \label{lemma:interp_x}
  Let $\VHx \subset H^1(\Omega)$ and $\VHl \subset H^{-1/2}(\Gamma)$ be finite dimensional and $\pi_{\Omega}: L^2(\Omega) \to \VHx$ be a linear operator. Then,
  there exists a linear operator $\Pi_x:  L^2(\Omega) \times L^2_{\rho_x}(\R^d)\to \VHx \times L^2_{\rho_x}(\R^d)$
  such that the following properties hold
  for $(u^h,u_\star^h):=\Pi_x(u,u_\star)$ with $u \in H^1(\Omega)$, $u_\star \in H^1_{\rho_x}(\R^d\setminus \Gamma)$ satisfying $u_\star|_\Omega = 0$ and $\tracejump{u_{\star}}= \gamma^- u$: 
  \begin{enumerate}[(i)]
  \item \label{it:interp_x_1} $\gamma^- u^h_{\star} \in (\VHl)^\circ$;
  \item  \label{it:interp_x_2} $\tracejump{u_{\star}^h}= \gamma^- u^h$;
  \item  \label{it:interp_x_3}
    If $\pi_{\Omega}$ is stable in the $H^1(\Omega)$-norm, then
    \begin{align*}
      \|u^h\|_{H^1(\Omega)}^2 + \|u^h_\star\|_{H^1_{\rho_x}(\R^d\backslash\Gamma)}^2
      \lesssim  \norm{u}_{H^1(\Omega)}^2 + \norm{u_\star}_{H^1_{\rho_x}(\R^d\backslash\Gamma)}^2. 
    \end{align*}
        If $\pi_{\Omega}$ is stable in the $L^2(\Omega)$-norm, and $u_\star \in L^2(\R^d)$ then
    \begin{align*}
      \|{u^h}\|_{L^2(\Omega)}^2 + \| {u^h_\star}\|_{L^2(\R^d)}^2
      &\lesssim  \norm{u}_{L^2(\Omega)}^2 + \norm{u_\star}_{L^2(\R^d)}^2.
    \end{align*}
  \item  \label{it:interp_x_4} There hold the approximation properties:  
    \begin{align*}
      \|u^h - u\|_{L^2(\Omega)}^2 + \|u_\star^h - u_\star\|_{L^2_{\rho_x}(\R^d\setminus \Gamma)}^2  
      &\lesssim \|u - \pi_{\Omega} u\|^2_{L^2(\Omega)}, \\
      \|u^h - u\|_{H^1(\Omega)}^2 + \|u_\star^h - u_\star\|_{H^1_{\rho_x}(\R^d\setminus \Gamma)}^2  
      &\lesssim \|u - \pi_{\Omega} u\|^2_{H^1(\Omega)}.
    \end{align*}
  \end{enumerate}
  \begin{proof}
   We note that a very similar operator is introduced in~\cite[Lemma 4.3]{schroedinger}.
    We define:
    \begin{align*}
      u^h:=\pi_{\Omega} u, \qquad
      u^h_\star:=u_\star+ \delta,
    \end{align*}
    where $\delta=-u_\star$ in $\Omega$ and $\delta:=-\mathcal{E}(u - \pi_{\Omega} u)$ in $\R^d \backslash\Omega$, where
    $\mathcal{E}: L^2(\Omega) \to L^2(\R^d)$ denotes the Stein extension operator
    \cite[Chapter VI.3]{stein70} that is stable both in $L^2(\Omega)$ and $H^1(\Omega)$.

    By construction, we have \eqref{it:interp_x_1}, since $u^h_\star=0$ in the interior.
    Since $\gamma^+ \mathcal{E}v =\gamma^- v$ due to the extension property, we
    get \eqref{it:interp_x_2} by 
    $$
    \llbracket{\gamma u_{\star}^h}\rrbracket = \gamma^+ u_\star - \gamma^+ \mathcal{E}(u - \pi_{\Omega} u) = \gamma^- u - \gamma^- u + \gamma^-\pi_\Omega u = \gamma^- u^h. 
    $$    
    The stability estimates follow from the stability of the extension operator and the assumed stabilities of $\pi_\Omega$ as 
    \begin{align*}
      \|{u^h}\|_{L^2(\Omega)}^2 + \| {u^h_\star}\|_{L^2(\R^d)}^2 &\lesssim  \norm{u}_{L^2(\Omega)}^2 + \norm{u_\star}_{L^2(\R^d)}^2 + \norm{\mathcal{E}(u - \pi_{\Omega} u)}_{L^2(\R^d)}^2  \\
      &\lesssim  \norm{u}_{L^2(\Omega)}^2 + \norm{u_\star}_{L^2(\R^d)}^2 + \norm{u - \pi_{\Omega} u}_{L^2(\Omega)}^2 \lesssim  \norm{u}_{L^2(\Omega)}^2 + \norm{u_\star}_{L^2(\R^d)}^2.
    \end{align*}
    The approximation property can be seen in a similar fashion using $\rho_x^{-2} <1$
    \begin{align*}
      \|u^h - u\|_{L^2(\Omega)}^2+ \|u_\star^h - u_\star\|_{L^2_{\rho_x}(\R^d)}^2
      &\leq
        \|\pi_\Omega u - u\|_{L^2(\Omega)}^2+ \|\mathcal{E}(u-\pi_\Omega u)\|_{L^2_{\rho_x}(\R^d)}^2 \\
     &\leq
        \|\pi_\Omega u - u\|_{L^2(\Omega)}^2+ \|\mathcal{E}(u-\pi_\Omega u)\|_{L^2(\R^d)}^2 
        \lesssim
        \|\pi_\Omega u - u\|_{L^2(\Omega)}^2.        
    \end{align*}
    The $H^1$-estimates follows analogously.
  \end{proof}
\end{lemma}

\begin{lemma}[Interpolation in $y$]
  \label{lemma:interp_y}
   Let $\YY \in (0,\infty)$ and $\UU^\YY$ solve \eqref{eq:truncated_problem_pw}.  
  Let $\TT_y$ be a geometric grid on $(0,\YY)$ with mesh grading factor $\sigma$,
  and $L$-refinement layers towards $0$ as given by \eqref{eq:geometric_mesh}. Let $\varepsilon>0$ be given by Proposition~\ref{prop:regularity}.
  Then, choosing $L=p$, 
  there exists an operator
  $\Pi_y: \HHwry \rightarrow \HHwry$
    such that $\Pi_y{\UU}(x,\cdot) \in \mathcal{S}^{p,1}(\TT_y)$ for almost all $x \in \R^d$,
    and such that the following estimate holds:
  \begin{align*}
    \int_{0}^{\YY}\int_{\R^d}{
    y^\alpha |\nabla(\UU^\YY - \Pi_y {\UU^\YY})|^2 dx dy }
    &\leq C e^{-2\kappa p} \YY^{2\varepsilon}. 
  \end{align*} 
  The constants $C,\kappa>0$ are independent of $p,\YY$.
\end{lemma}
\begin{proof}
We use the $hp$-interpolation operator from \cite[Sec.~5.5.1]{tensor_fem} for $\Pi_y$. This operator is constructed on a geometric mesh in an element-by-element way. On the first element a linear interpolant in $\sigma^L/2$ and $\sigma^L$ is used, while the remaining elements are mapped to the reference element, on which a polynomial approximation operator that has exponential convergence properties (in the polynomial degree) for analytic functions is used. 

For the operator on the reference element, we take the  Bab\u{u}ska-Szab\'o polynomial approximation operator $\widehat{\Pi}_p$ on $(-1,1)$  defined as 
\begin{align*}
\widehat{\Pi}_p v (y) := v(-1) +  \int_{-1}^y \Pi^{L^2}_{p-1} v'(t) dt,
\end{align*}
where $\Pi^{L^2}_{p-1} : L^2(-1,1) \rightarrow P_{p-1}$ denotes the $L^2$-orthogonal projection, see e.g. \cite[Exa.~3.17]{ApeMel15}. By construction, this operator has the commutator property
\begin{align*}
(\widehat{\Pi}_p v)' = \Pi^{L^2}_{p-1} v'.
\end{align*}
Regularity in countably normed spaces gives exponential error bounds for $\Pi_y$, see \cite[Lem.~11]{tensor_fem}. In fact, for functions in $\mathcal{B}^1_{\varepsilon,0}(C,K;\YY,L^2(\R^d))$, one obtains a bound in $L^2(y^\alpha,\R^d\times (0,\YY))$.
Consequently, we can employ Proposition~\ref{prop:regularity} to obtain $\nabla_x \UU^\YY \in \mathcal{B}^1_{\varepsilon,0}(C,K;\YY,L^2(\R^d))$ and together with \cite[Lem.~11(i)]{tensor_fem} this gives the error estimate 
\begin{align*}
    \int_{0}^{\YY}{
    y^\alpha \|\nabla_x\UU^\YY(\cdot,y) - \Pi_y {\nabla_x\UU^\YY(\cdot,y)}\|_{L^2(\R^d)}^2 dy }
    &\leq C e^{-2\kappa p} \YY^{2\varepsilon}
  \end{align*}
  for a constant $\kappa >0$. Interchanging $\Pi_y$ and $\nabla_x$ gives the estimate for the $x$-derivatives. 
  
For the $y$-derivatives the situation is a bit more involved, as the same argument can not be made as $\Pi_y$ and $\partial_y$ do not commute. 
\cite[Lem.~11(ii)]{tensor_fem} gives an exponentially convergent error bound for the $y$-derivative provided $\UU^\YY \in \mathcal{B}^2_{\varepsilon,0}(C,K;\YY,L^2(\R^d))$ (essentially meaning $\partial_y \UU^\YY \in \mathcal{B}^1_{\varepsilon,0}(C,K;\YY,L^2(\R^d))$ and $\UU^\YY \in L^2(y^\alpha,\R^d\times(0,\YY))$).
However, in our setting, the requirement $\UU^\YY \in L^2(y^\alpha,\R^d\times(0,\YY))$ does not hold. 
Nonetheless, we have Corollary~\ref{cor:regularity_y} giving $\partial_y \UU^\YY \in \mathcal{B}^1_{\varepsilon,0}(C,K;\YY,L^2(\R^d))$, which is enough to regain the exponential estimate as seen in the following. 

On the first element $(0,\sigma^L) \in \mathcal{T}_y$, the definition of the piecewise linear interpolation gives 
\begin{align*}
\partial_y \Pi_y v(y) = \frac{v(\sigma^L)-v(\sigma^L/2)}{\sigma^L/2} = \frac{2}{\sigma^L}\int_{\sigma^L/2}^{\sigma^L} \partial_y v(y) dy, 
\end{align*}
which is nothing else than the $L^2$-orthogonal projection of $\partial_y v$ on $(\sigma^L/2,\sigma^L)$.
By choice of the Bab\u{u}ska-Szab\'o operator and denoting by $\widetilde\Pi^{L^2}_{p-1}$ the mapped $L^2$-projection onto an element in $\mathcal{T}_y$, we have due to the commutator property and the preceding discussion
\begin{align*}
\partial_y(\Pi_y \UU^\YY)|_{\R^d\times K_i} = \widetilde\Pi^{L^2}_{p-1} \partial_y \UU^\YY|_{\R^d\times K_i} \in L^2(y^\alpha,\R^d\times K_i) \qquad \forall K_i \in \mathcal{T}_y
\end{align*}
since $\partial_y\UU \in L^2(y^\alpha,\R^d\times K_i)$, which implies that $\partial_y \Pi_y \UU^\YY \in L^2(y^\alpha,\R^d\times (0,\YY))$. The error estimate for the $y$-derivative follows from scaling arguments.
 More precisely, we decompose 
\begin{align*}
\norm{\partial_y(\UU^\YY-\Pi_y\UU^\YY)}_{L^2(y^\alpha,\R^d \times (0,\YY))}^2 = \sum_{K_i \in \mathcal{T}_y} \norm{\partial_y(\UU^\YY-\Pi_y\UU^\YY)}_{L^2(y^\alpha,\R^d \times K_i)}^2,
\end{align*}
where $K_i =(x_i,x_{i+1})$.
Using a Hardy inequality, one obtains a bound for the approximation error on the first element using second derivatives only; see \cite[Lem.~15]{tensor_fem}. Together with a scaling argument this leads to
\begin{align*}
\norm{\partial_y(\UU^\YY-\Pi_y\UU^\YY)}_{L^2(y^\alpha,\R^d \times (0,\sigma^L))}^2 \lesssim \sigma^{\varepsilon L}\norm{\partial_{yy}\UU^\YY}_{L^2(y^{\alpha+2-2\varepsilon},\R^d \times (0,\sigma^L))}^2. 
\end{align*}
By Corollary~\ref{cor:regularity_y} we can bound the right-hand side. 
For the remaining elements, we employ a scaling argument from
\cite[Thm.~3.13]{ApeMel15}. 
Denoting by $h_{K_i}$ the diameter of $K_i$, we infer $y \sim h_{K_i}$ on $K_i$ for $i>0$.
For any univariate function $v$  satisfying $\norm{y^{\ell-\varepsilon}v^{(\ell+1)}}_{L^2(y^\alpha,(0,\YY))} < C K^{\ell} \ell!$ for all $\ell \in \N_0$  there holds
\begin{align}\label{eq:PIy_tmp1}
\norm{\widehat{v}^{(\ell+1)}}_{L^2(-1,1)}^2 &= \frac{2}{h_{K_i}} h_{K_i}^{2(\ell+1)}  \norm{v^{(\ell+1)}}_{L^2(K_i)}^2 \lesssim 
h_{K_i}^{2\varepsilon-\alpha+1}  \norm{y^{\ell-\varepsilon}v^{(\ell+1)}}_{L^2(y^{\alpha},K_i)}^2 \nonumber\\
&\lesssim h_{K_i}^{2\varepsilon - \alpha+1} K^{\ell}\ell!,
\end{align}
where  $\widehat v$ is the pull-back of $v$ to the reference element. The exponential approximation properties of the Bab\u{u}ska-Szab\'o polynomial approximation operator then provides
\begin{align}\label{eq:PIy_tmp2}
\norm{\widehat{v}-\widehat \Pi_p \widehat v}_{H^1(-1,1)}^2 \lesssim h_{K_i}^{2\varepsilon - \alpha+1} e^{-\kappa p}.
\end{align}
Together with
\begin{align*}
\norm{(v-\Pi_y v)'}_{L^2(y^\alpha,K_i)}^2 \lesssim h_{K_i}^{\alpha-1} \norm{(\widehat v-\widehat\Pi_p \widehat v)'}_{L^2(-1,1)}^2 ,
\end{align*}
we can employ \eqref{eq:PIy_tmp2} for $v(y) = \UU(y,\cdot)$ and square integrate over $\R^d$, noting that \eqref{eq:PIy_tmp1} holds due to Corollary~\ref{cor:regularity_y}. 
Summing over $i$ and using $\sum_i h_{K_i}^{2\varepsilon} \lesssim \YY^{2\varepsilon}$ shows the claimed estimate.

Finally, to show that the operator does indeed map to $\HHwry$, we note that
  by the previous considerations we have $\partial_y \Pi_y \UU \in L^2(y^\alpha,\R^d \times (0,\YY))$
  as well as $\Pi_y \UU(\cdot,y)=\UU(\cdot,y) \in L_{\rho_x}^2(\R^d)$ for certain values $y \in (0,\YY)$
  where it is interpolatory. By the fundamental theorem of calculus, this is sufficient to show that $\Pi_y \UU \in L^2_\rho(y^\alpha,\R^d\times (0,\YY))$.
\end{proof}



We can now define an interpolation operator acting on both $x$ and $y$ in a tensor product fashion.
In order to keep notation compact,  we write  $\|\cdot\|_{L^2}$ for the $L^2(\Omega)\times L^2(\R^d \backslash \Gamma)$-norm and 
 $\|\cdot\|_{H^1_{\rho_x}}$ for the $H^1(\Omega)\times H^1_{\rho_x}(\R^d\backslash\Gamma)$-norm. 

\begin{lemma}[Tensor approximation]
  \label{lemma:tensor_interpolation}
  Fix $\YY \in (0,\infty)$ and let $\UU = (\UU_\Omega,\UU_\star) \in \HH_\YY$. Define $\Pi(\UU_\Omega,\UU_\star):=\Pi_{x} \otimes \Pi_{y}(\UU_\Omega,\UU_\star) \in \HH_{h,\YY}$ with the operators $\Pi_x$ from Lemma~\ref{lemma:interp_x} and $\Pi_y$ from Lemma~\ref{lemma:interp_y}. Assume that the operator $\pi_\Omega$ in the definition of $\Pi_x$ is both $L^2$- and $H^1$-stable. 
  Then, the following approximation estimate holds
  \begin{align*}
    \Big\|\UU - \Pi \UU\|_{\HH_\YY}^2
    &\lesssim
      \int_{0}^{\YY}
      {
      y^\alpha \Big( \big\|\nabla (1-\Pi_y)\UU(y)\big\|_{L^2}^2  +
      \big\|\nabla (1-\pi_\Omega) \UU_\Omega(y)\|_{L^2(\Omega)}^2   \Big)\,dy. 
      } 
  \end{align*}
\end{lemma}
\begin{proof}
By the Poincar\'e inequality \eqref{eq:my_poincare} and the trace inequality \eqref{eq:trace}, we only have to estimate the gradient norms. We start with the $x$-derivatives. Employing the $H^1$-stability
      and approximation properties of $\Pi_x$
      from Lemma~\ref{lemma:interp_x}~(\ref{it:interp_x_3}) and (\ref{it:interp_x_4}) gives 
  \begin{align*}
    \int_{0}^{\YY}
    {
      y^\alpha \big\|\nabla_x ( \UU - \Pi \UU)\|_{L^2}^2 \,dy
    }
    &\lesssim
    \int_{0}^{\YY}
      y^\alpha \big\|\nabla_x \UU -  \nabla_x (\Pi_x \otimes I) \UU\|_{L^2}^2 dy
       \\
      &\qquad +
        \int_{0}^{\YY}  y^\alpha \big\|\nabla_x (\Pi_x \otimes I) \UU - \nabla_x (\Pi_x \otimes \Pi_y) \UU\|_{L^2}^2 \,dy
     \\    
    &\lesssim
      \int_{0}^{\YY}
    {
      y^\alpha \big\| (I- \pi_\Omega) \UU_{\Omega}(y)\big\|_{H^1(\Omega)}^2 dy
      +     \int_{0}^{\YY} y^\alpha\big\| (I- \Pi_y) \UU(y)\big\|_{H^1_{\rho_x}}^2\,dy
      }.
      \end{align*}
Employing again Poincar\'e inequalities, we can reduce the right-hand side to norms of derivatives only.
      For the $y$-derivative, we proceed similarly using the $L^2$- stability and approximation properties of $\Pi_x$ 
      \begin{align*}
        \int_{0}^{\YY}
        {
        y^\alpha \big\|\partial_y (\UU -  \Pi \UU) \|_{L^2}^2 \,dy
        }
        &\lesssim
          \int_{0}^{\YY}
          y^\alpha \big\|\partial_y \UU - \partial_y (\Pi_x \otimes I) \UU\|_{L^2}^2 dy \\
          & \qquad +
           \int_{0}^{\YY} y^\alpha \big\|\partial_y (\Pi_x \otimes I) \UU - \partial_y (\Pi_x \otimes \Pi_y) \UU\|_{L^2}^2 \,dy
           \\
        &\lesssim
           \int_{0}^{\YY}
          {
          y^\alpha \| (1- \pi_\Omega) \partial_y \UU_\Omega(y) \|_{L^2(\Omega)}^2 +
          y^\alpha \big\|\partial_y(I - \Pi_y) \UU(y)\|_{L^2}^2 \,dy
          },
      \end{align*}
      which finishes the proof.
\end{proof}

In order to obtain a best-approximation estimate for the semi-discretization, we observe that the difference $\UU-\UU_h$ satisfies some form of Galerkin orthogonality.

\begin{lemma}[Galerkin orthogonality]
  \label{lemma:galerkin_ortho}
  Let $\YY>0$,  $\UU^\YY = (\UU^\YY_\Omega,\UU^\YY_\star) \in \HH_\YY$ be the solution of~\eqref{eq:truncated_BLF_eq}
  and $\UU_h^\YY \in \HH_{h,\YY}$ solve \eqref{eq:truncated_BLF_eq_discrete}.
  Then, for all $\VV_h^\YY = (\VV_\Omega^\YY,\VV_\star^\YY) \in \HH_{h,\YY}$ and $\lambda_h\in \VHl$, there holds
  \begin{align*}
    B^\YY(\UU^\YY - \UU^\YY_h,\VV^\YY_h)
    &= \int_{0}^{\YY}{y^\alpha\big<\normaljumps{\UU^\YY_\star}- \lambda_h, \gamma_{\Gamma}^- {\VV_\star^\YY}\big>_{L^2(\Gamma)}\,dy}.
  \end{align*}
\end{lemma}
\begin{proof}
  Compared to ``standard'' Galerkin orthogonality, we observe
  that $\VV_h^\YY$ is not an admissible
  test function in~\eqref{eq:truncated_BLF_eq} due to the weak condition of $\gamma_{\Gamma}^-\VV_\star^\YY \in (\VHl)^\circ$
  compared to $\gamma_{\Gamma}^-\VV_\star^\YY \in (H^{-1/2}(\Gamma))^\circ=\{0\}$. Also,
   if we work in the $\HHwr$-setting (i.e. working with global functions instead of pairs),
  the test function $\VV_\Omega^\YY \chi_{\Omega}+\VV_{\star}^\YY \chi_{\Omega^c}$ is not continuous along $\Gamma$
  due to a possible jump of size $\gamma_{\Gamma}^-\VV_\star^\YY$. However, if we use the pointwise equation
  (\ref{eq:truncated_problem_pw})  and integrate back by parts, we get that
  \begin{align*}
    B^\YY(\UU^\YY - \UU_h^\YY,\VV_h^\YY)
    &= \int_{0}^{\YY}{y^\alpha\big<\normaljumps{\UU^\YY_\star}, \gamma_{\Gamma}^-{\VV_\star^\YY}\big>_{L^2(\Gamma)}\,dy}.
  \end{align*}
Since $\langle{\lambda_h,\gamma_{\Gamma}^-\VV_\star^\YY}\rangle_{L^2(\Gamma)}$ vanishes due to the requirement in $\gamma_{\Gamma}^-\VV_\star^\YY \in (\VHl)^\circ$,
we can subtract such a term from the right-hand side without changing the equality, which shows the
stated Galerkin orthogonality.  
\end{proof}

Finally, we are in position to show our main result, Theorem~\ref{prop:bestApproximation}, by combining the decay estimate with the previous two lemmas.

\begin{proof}[Proof of Theorem~\ref{prop:bestApproximation}]
We start with the triangle inequality 
$$
 \|\UU - \UU_h^\YY\|_{\HH_\YY} \leq     \|\UU - \UU^\YY\|_{\HH_\YY} +     \|\UU^\YY - \UU_h^\YY\|_{\HH_\YY}.
 $$
For the first term, we use the decay properties of  Proposition~\ref{prop:truncation_error} to obtain 
$$
\|\UU - \UU^\YY\|_{\HH_\YY}  \lesssim  \|\UU_{\Omega}-\UU_\Omega^\YY\|_{H^1(y^\alpha, \Omega\times (0,\YY))}
      + \|\UU_{\star}-\UU_\star^\YY\|_{H^1_\rho(y^\alpha, \R^d \setminus \Gamma\times (0,\YY))}  \lesssim \YY^{-\mu/2}\norm{f}_{L^2(\Omega)}.
$$
For the second term, we employ the coercivity of Theorem~\ref{prop:cont_well_posedness},
the Galerkin orthogonality of Lemma~\ref{lemma:galerkin_ortho}, $\UU^\YY_\star|_\Omega = 0$, and a trace inequality for $\Omega$, which gives for arbitrary $\lambda_h \in \VHl$ and $\VV_h^\YY = (\VV_\Omega^\YY,\VV_\star^\YY)  \in  \HH_{h,\YY}$ that
  \begin{align*}
    \|\UU^\YY - \UU_h^\YY\|_{\HH_\YY}^2
    &\lesssim B^\YY(\UU^\YY-\UU_h^\YY,\UU^\YY-\UU_h^\YY) \\
    &  = B^\YY(\UU^\YY-\UU_h^\YY,\UU^\YY-\VV_h^\YY) +
      \int_{0}^{\YY}{y^\alpha
      \big<\normaljumps{\UU^{\YY}_\star}-\lambda_h, \gamma_{\Gamma}^-(\UU_{h,\star}^\YY -\VV_{\star}^\YY)\big>_{L^2(\Gamma)}\,dy} \\
    &\lesssim
      \varepsilon \|{\UU^\YY-\UU_h^\YY}\|_{\HH_\YY}^2 + \varepsilon^{-1}\|{\UU^\YY-\VV_h^\YY}\|_{\HH_\YY}^2  \\
    &\qquad  +\int_{0}^{\YY}{y^\alpha \|\normaljumps{\UU^{\YY}_{\star}} -\lambda_h\|_{H^{-1/2}(\Gamma)}\|\gamma_{\Gamma}^-(\UU_{h,\star}^\YY -\VV_{\star}^\YY)\|_{H^{1/2}(\Gamma)} dy} \\
    &\lesssim
      \varepsilon \|{\UU^\YY-\UU_h^\YY}\|_{\HH_\YY}^2 + \varepsilon^{-1}\|{\UU^\YY-\VV_h^\YY}\|_{\HH_\YY}^2  \\
    &\qquad  +\varepsilon^{-1}\int_{0}^{\YY}{y^\alpha \|\normaljumps{\UU^{\YY}_{\star}} -\lambda_h\|^2_{H^{-1/2}(\Gamma)} dy}
      +\varepsilon\|\UU_h^\YY -\VV_h^\YY\|_{\HH_\YY}^2 \\
    &\lesssim
      2\varepsilon \|{\UU^\YY-\UU_h^\YY}\|_{\HH_\YY}^2 + (\varepsilon+\varepsilon^{-1})\|{\UU^\YY-\VV_h^\YY}\|_{\HH_\YY}^2
      + \varepsilon^{-1}\int_{0}^{\YY}{y^\alpha \| \normaljumps{\UU^{\YY}_{\star}} -\lambda_h\|_{H^{-1/2}(\Gamma)}^2 dy}.
  \end{align*}
  Taking $\varepsilon$ sufficiently small and absorbing the first term in the left-hand side gives 
  $$
  \|\UU^\YY - \UU_h^\YY\|_{\HH_\YY}^2 \lesssim \|\UU^\YY-\VV_h^\YY\|_{\HH_\YY}^2 +
  \int_{0}^{\YY}{y^\alpha \|\normaljumps{\UU^{\YY}_{\star}} -\lambda_h\|^2_{H^{-1/2}(\Gamma)} dy}.
  $$
  As $\VV_h^\YY \in  \HH_{h,\YY}$ was arbitrary, we can take $\VV_h^\YY = \Pi(\UU^\YY_\Omega,\UU^\YY_\star) \in  \HH_{h,\YY}$ with the operator $\Pi$ of Lemma~\ref{lemma:tensor_interpolation}. Then,  Lemma~\ref{lemma:tensor_interpolation} together with the approximation properties of the $hp$-interpolation in $\YY$ gives
\begin{align*}
\|\UU^\YY- \Pi \UU^\YY\|_{\HH_\YY}^2 &\lesssim       \int_{0}^{\YY}  y^\alpha \Big( \big\|\nabla (1-\Pi_y)\UU^\YY(y)\big\|_{L^2}^2  +
      \big\|\nabla (1-\pi_\Omega) \UU_\Omega^\YY(y)\|_{L^2(\Omega)}^2   \Big)\,dy  \\
     &\lesssim \YY^{2\varepsilon} e^{-2\kappa p} + \int_{0}^{\YY}  y^\alpha  \big\|\nabla (1-\pi_\Omega) \UU_\Omega^\YY(y)\|_{L^2(\Omega)}^2 \,dy.
\end{align*}
 Combining all estimates gives the stated result.
\end{proof}

Finally, we present the proof of Corollary~\ref{cor:convergenceRates} that gives first order convergence for a specific choice of discrete spaces.

\begin{proof}[Proof of Corollary~\ref{cor:convergenceRates}]
  Employing \cite[Pro.~2.8]{part_one} -- which with the same techniques also holds for $\YY < \infty$ and a constant independent of $\YY$ -- together with the assumptions on $ \Omega, \mathfrak{A},$ and $f$,  we obtain control of second order $x$-derivatives of $\UU^\YY$.

As $\lambda:=\partial_\nu^+ \UU^{\YY}_\star$, it is piecewise smooth, depending on the
  regularity of $\UU^{\YY}_\star$.
  For $m=0,1$, denoting by $\pi_{L^2}$ the $L^2$-projection onto $\mathcal{S}^{0,0}(\TT_{\Gamma})$ and using a trace estimate, it holds that
  \begin{align*}
     \|\lambda(y) -\lambda_h(y)\|_{H^{-1/2}(\Gamma)}^2
    &\!\!\!\stackrel{\text{\cite[\tiny  Thm~4.1.33]{sauter_schwab}}}{\lesssim}\!\!
      h^{1/2} \|\lambda(y) - \pi_{L^2}\lambda(y)\|_{L^2(\Gamma)}^2  \\
     & \! \!\!\stackrel{\text{\cite[\tiny Prop~4.1.31]{sauter_schwab}}}\lesssim \!\! h^{1/2} h^{m}\sum_{K \in \TT_{\Gamma}}
    {\|\lambda(y)\|_{H^{m}(K)}^2} \lesssim h^{1/2} h^{m} \|\UU^{\YY}_{\star}(y)\|_{H^{m+3/2}(B_R(0)\backslash\Gamma)}^2.
  \end{align*}
  Interpolating between $m=0$ and $m=1$ gives
  \begin{align*}
    \|\lambda(y) -\lambda_h(y)\|_{H^{-1/2}(\Gamma)}^2    
    &\lesssim h \|\UU^{\YY}_{\star}(y)\|_{H^2(B_R(0)\backslash\Gamma)}^2.
  \end{align*}
Multiplying with $y^\alpha$ and integrating with respect to $y$ then controls the second term on the right-hand side of Theorem~\ref{prop:bestApproximation}.
For the first term, the approximation properties of the Scott-Zhang projection together with control of the second order $x$-derivatives gives first order convergence in $h$. Finally, the last two terms in Theorem~\ref{prop:bestApproximation} can also be bounded by $Ch$ by choice of $\YY$ and $p$.
\end{proof}

\section{Numerics}
\label{sec:numerics}
In this section, we present two numerical examples to underline the a-priori estimates of Theorem~\ref{prop:bestApproximation} and Corollary~\ref{cor:convergenceRates}. 
As previously already mentioned, a nice feature of our numerical scheme is that software packages developed for integer order differential operators can be employed directly. 
As such, we implement our method  based on a coupling of the libraries NGSolve (\cite{ngsolve}, for the FEM-part) and Bempp-cl (\cite{bempp-cl}, for the BEM-part) libraries.
\bigskip

In order to validate our numerical method, we consider the case $s=0$
and the standard fractional Laplacian, i.e., $\mathfrak{A}=I$.
In this case a representation formula is available
from \cite{caffarelli_silvestre}. In fact, the fundamental solution for the fractional Laplacian is given by
\begin{align*}
\Psi(x) := \frac{C_{d,\beta}}{\abs{x}^{d-2\beta}} \qquad x \in \R^d\backslash \{0\}, \; d\neq 2\beta
\end{align*}
with $C_{d,\beta}:= \frac{\Gamma(d/2-\beta)}{2^{2\beta}\pi^{d/2}\Gamma(\beta)}$.
Thus, for $f \in C_{0}^{\infty}(\Omega)$ we can write
$$
u(x)=C_{d,\beta}\int_{\R^d}{\frac{f(y)}{\abs{x-y}^{d-2\beta}} \;dy}.
$$

We then calculate $u(x)$ at random sampling points $x_j$ using spherical coordinates
and Gauss-Jacobi numerical integration to
deal with the singularity at $r=|x-x_j|=0$ as well as standard Gauss-Quadrature for the other coordinate directions.

In order to compute the energy error, 
  we compute the energy differences.
  For standard FEM with bilinear form $a(\cdot,\cdot)$ and right-hand side $f(\cdot)$, it is well known that one can compute the energy error by the identity $\|u-u_h\|^2_{E}=a(u,u)-a(u_h,u_h)=f(u)-f(u_h).$
  Due to the more complicated form of our method, most notably the presence of the cutoff error, such an identity does not hold exactly. Nevertheless, we expect the following identity to hold approximately
  $$
  \|\UU - \UU_{h}^{\YY}\|_{\HHwy}^2 \approx (f,\trace \UU)_{L^2(\Omega)} - (f,\trace \UU^{\YY}_h)_{L^2(\Omega)}.
  $$
  We now further replace the unknown value $(f,\trace \UU)_{L^2(\Omega)}$ by the
  extrapolation from $(f,\trace \UU^{\YY}_h)_{L^2(\Omega)}$ for different refinements
  using Aitken's $\Delta^2$-method.  This will be our approximation of the true energy error.
  For the $L^2$-error, we use the approximation $\UU^{\YY}_h$ on the finest grid as our standin
  for the exact solution compare it to the other approximations
  by computing the $L^2$-difference of the traces at $y=0$ using  Gauss quadrature.

For the geometry, we used the unit cube $\Omega:=[-1,1]^3$.
In the bounded domain $\Omega$, we use piecewise linear Lagrangian finite elements on a quasi-uniform mesh of maximal mesh width $h$.

In Figure~\ref{fig:convergence_rates}, we study the convergence  of the proposed fully discrete method
as we reduce the mesh size. In order to reduce all the error contributions, 
we choose the cutoff point $\YY=h^{-\frac{2}{1+\alpha}}$ which gives $\bigO(h)$
for the cutoff error in  Proposition~\ref{prop:truncation_error}. Since the convergence with respect to
the polynomial degree is exponential (but with unknown explicit rate), we use
$p:=\operatorname{round}(2m \log(m+1))$ where $m$ is the number of uniform $h$-refinements. This gives a decrease of the
$y$-discretization error which is faster than $\bigO(h)$. Overall we
expect the energy error to behave like $\bigO(h)$ by Corollary~\ref{cor:convergenceRates}.
For the pointwise and $L^2$-errors we did not establish a rigorous theory in this work. Nonetheless, Figure~\ref{fig:convergence_rates} shows convergence rates for these error measures
of roughly order $\bigO(h^2)$.

\begin{figure}
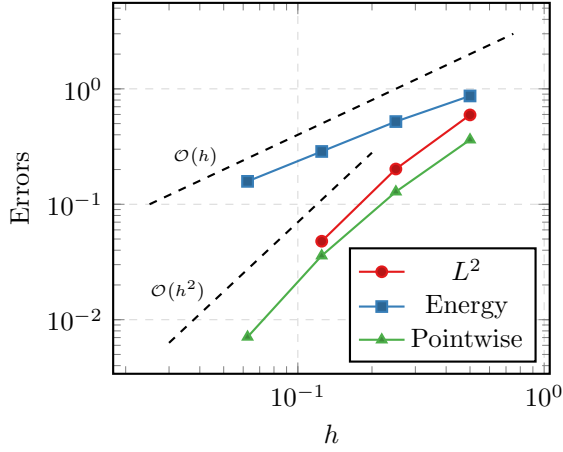
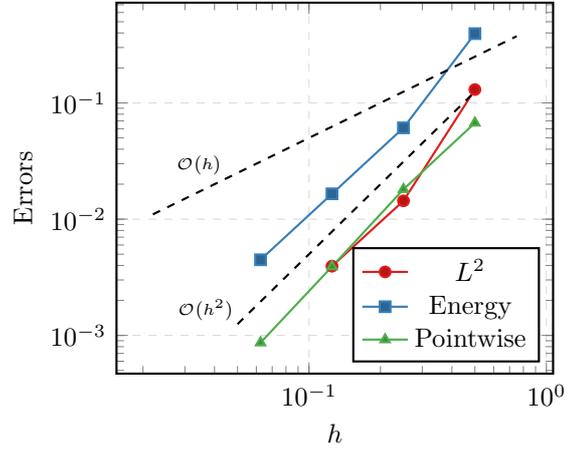
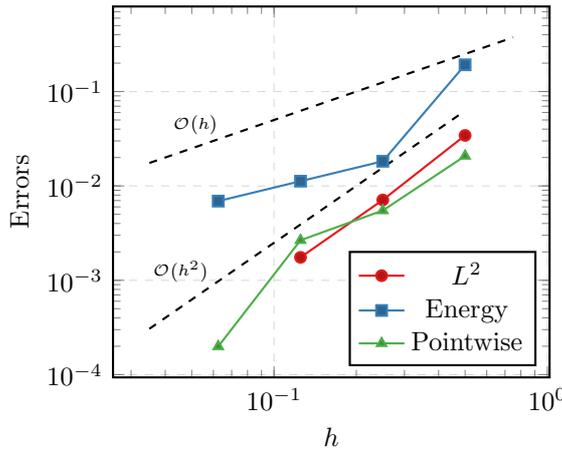
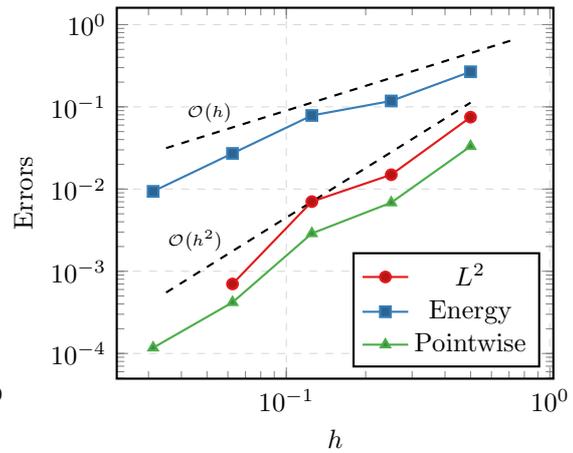

  \centering
  \begin{subfigure}{0.46\textwidth}
    \includeTikzOrEps{convergence_01}
    \vspace{-5mm}
  \caption{Convergence for $\beta=0.1$}
  \label{fig:convergence1}
\end{subfigure}
\begin{subfigure}{0.46\textwidth}
  \includeTikzOrEps{convergence_03}
  \vspace{-5mm}
  \caption{Convergence for $\beta=0.3$}
  \label{fig:convergence3}
\end{subfigure} \\\vspace{5mm}
  \begin{subfigure}{0.46\textwidth}
    \includeTikzOrEps{convergence_05}
    \vspace{-5mm}
  \caption{Convergence for $\beta=0.5$}
  \label{fig:convergence5}
\end{subfigure}
\begin{subfigure}{0.46\textwidth}
  \includeTikzOrEps{convergence_07}
      \vspace{-5mm}
  \caption{Convergence for $\beta=0.7$}
  \label{fig:convergence7}
\end{subfigure}
\caption{Convergence of our discrete approximation to the exact solution for different fractional powers $\beta$ in different norms.}
\label{fig:convergence_rates}
\end{figure}

As a second numerical example, we consider as the domain $\Omega$ the unit sphere in $\R^3$.
Instead of using the standard Laplacian with constant coefficients, we consider
the following diffusion parameter and right-hand side:
\begin{align*}
  \mathfrak{A}(x):=\begin{cases}
    1+|x|(1-|x|) & \text{for }|x|<1 \\
    1 & \text{for } |x|\geq 0,
  \end{cases},
        \qquad \text{and}\qquad
        f(x):=\begin{cases}
         |x|(1-|x|) & \text{for } |x|<1 \\
         0 & \text{for } |x|\geq 0,
       \end{cases}
\end{align*}
(with the slight abuse of notation of making $\mathfrak{A}(x)$ scalar valued).
Since the coefficients are globally continuous, and we are working with lowest order
elements, by Corollary~\ref{cor:convergenceRates} we expect to obtain first order convergence. Figure~\ref{fig:convergence_rates_nc} supports the theoretical results.
Since in this case the fundamental solution is not available, we can not compute the pointwise error, but looking at the extrapolated energy and $L^2$-errors we get the optimal rates.

\begin{figure}
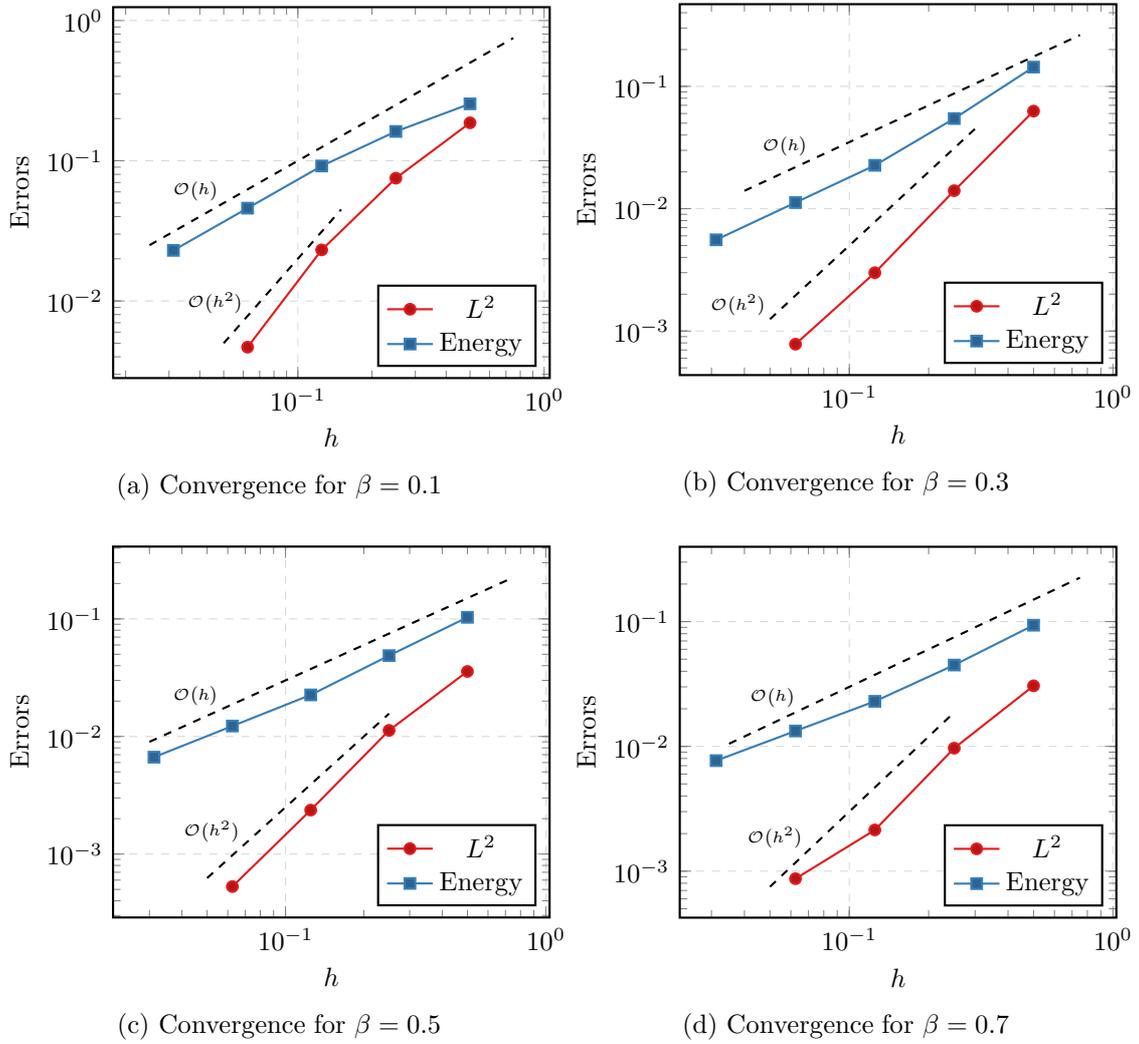

  \centering
  \begin{subfigure}{0.46\textwidth}
    \includeTikzOrEps{convergence_nonconst_01}
    \vspace{-5mm}
  \caption{Convergence for $\beta=0.1$}
  \label{fig:convergence_nc1}
\end{subfigure}
\begin{subfigure}{0.46\textwidth}
  \includeTikzOrEps{convergence_nonconst_03}
  \vspace{-5mm}
  \caption{Convergence for $\beta=0.3$}
  \label{fig:convergence_nc3}
\end{subfigure} \\\vspace{5mm}
  \begin{subfigure}{0.46\textwidth}
    \includeTikzOrEps{convergence_nonconst_05}
    \vspace{-5mm}
  \caption{Convergence for $\beta=0.5$}
  \label{fig:convergence_nc5}
\end{subfigure}
\begin{subfigure}{0.46\textwidth}
  \includeTikzOrEps{convergence_nonconst_07}
      \vspace{-5mm}
  \caption{Convergence for $\beta=0.7$}
  \label{fig:convergence_nc7}
\end{subfigure}
\caption{Convergence of our discrete approximation to the exact solution
  for different fractional powers $\beta$ in different norms, the non-constant coefficients case.}
\label{fig:convergence_rates_nc}
\end{figure}

\paragraph{Acknowledgments:} 
A.R. gladly acknowledges
financial support  by the Austrian Science Fund (FWF) through the 
project P~36150.

\bibliographystyle{alphaabbr}
\bibliography{literature}

\end{document}